\documentclass[a4paper,10pt]{amsart}
\usepackage{dsfont}
\usepackage{amsmath,amsthm,amssymb,tabu}
\usepackage{mathtools,enumitem}
\usepackage[all]{xy}
\usepackage[usenames,dvipsnames]{color}
\usepackage{listings} %for copy-pasting code
\usepackage{graphicx}
\usepackage[table]{xcolor}
\usepackage{hyperref}
\hypersetup{
	colorlinks=true,
	linkcolor=teal,
	filecolor=teal,
	urlcolor=teal,
	citecolor=teal
}
\usepackage[noadjust]{cite}
\usepackage{tikz,tikz-cd,tikz-layers}
\usetikzlibrary{shapes.geometric,calc,arrows}
\usepackage{wrapfig}
\usepackage{mdframed}
\usepackage{float}
\usepackage{diagbox,makecell,multirow}

\usepackage[iso-8859-7]{inputenc}

\newtheorem{theorem}{Theorem}[section]
\newtheorem{corollary}[theorem]{Corollary}
\newtheorem{lemma}[theorem]{Lemma}
\newtheorem{proposition}[theorem]{Proposition}
\newtheorem{remark}[theorem]{Remark}
\newtheorem{definition}[theorem]{Definition}

\newtheorem{notation}[theorem]{Notation}

\usepackage{adjustbox}
\def\max{\operatorname{max}}

\def\deg{\operatorname{deg}}

\def\g{\operatorname{g}}

\def\Proj{\operatorname{Proj}}

\def\Div{\operatorname{Div}}

\def\WDiv{\operatorname{WDiv}}
\def\CDiv{\operatorname{CDiv}}

\def\N{\operatorname{N}}
\def\NE{\operatorname{NE}}

\def\Exc{\operatorname{Exc}}
\def\Supp{\operatorname{Supp}}
\def\Bl{\operatorname{Bl}}

\def\Pic{\operatorname{Pic}}

\def\Nef{\operatorname{Nef}}
\def\Mov{\operatorname{Mov}}

\def\Ext{\operatorname{Ext}}
\def\Eff{\operatorname{Eff}}

\newcommand{\OO}{\mathcal{O}}

\newcommand{\NN}{\mathbb{N}}
\newcommand{\ZZ}{\mathbb{Z}}
\newcommand{\QQ}{\mathbb{Q}}
\newcommand{\RR}{\mathbb{R}}
\newcommand{\CC}{\mathbb{C}}
\newcommand{\PP}{\mathbb{P}}
\newcommand{\FF}{\mathbb{F}}

\newcommand{\isom}{\cong}
\newcommand{\defeq}{\vcentcolon=}
\newcommand{\ra}{\Rightarrow}

\newcommand{\map}{\smash{\xymatrix@C=0.5cm@M=1.5pt{ \ar[r]& }}}
\newcommand{\maps}[1]{\smash{\xymatrix@C=0.5cm@M=1.5pt{ \ar[r]^{#1}& }}}
\newcommand{\rmap}{\smash{\xymatrix@C=0.5cm@M=1.5pt{ \ar@{-->}[r]& }}}
\newcommand{\psmap}{\smash{\xymatrix@C=0.5cm@M=1.5pt{ \ar@{..>}[r]& }}}
\newcommand{\norm}[1]{\left\lVert#1\right\rVert}

\usepackage{subfiles} % Best loaded last in the preamble

\title{Sarkisov Links with centres space curves on smooth cubic surfaces}
\author{Sokratis Zikas}
\date{\today}
\address{Universit\"at Basel, Departement Mathematik und Informatik, Spiegelgasse 1, CH--4051 Basel, Switzerland}
\email{sokratis.zikas@unibas.ch}

\begin{document}

\maketitle	

\begin{abstract}
	We construct and study Sarkisov links obtained by blowing up smooth space curves lying on smooth cubic surfaces. 
	We restrict our attention to the case where the blowup is not weak Fano. Together with the results of \cite{WeakFanos} which cover the weak Fano case, we provide a classification of all such curves.
	This is achieved by computing all curves which satisfy certain necessary criteria on their multisecant curves and then constructing the Sarkisov link step by step.
\end{abstract}

\section{Introduction}

The Sarkisov program is a central tool in the study of birational relations between Mori fibre spaces. It was first proved by Corti in \cite{Corti} for dimension $3$ and by Hacon and McKernan in \cite{HM13} in the general case. In recent years there has been an increase in the number applications of it, mostly to the study of birational automorphisms of certain Mori fibre spaces. \cite{MR3604978}, \cite{MR3505644}, \cite{BLZ19} and \cite{MR3807895} are but a few papers on the subject. Thus, it is natural to study and try to  classify links involving a fixed Mori fibre space.

In this paper we will study links where one of the two Mori fibre spaces is $\PP^3 \map pt$. Notice that since the Picard rank of $\PP^3$ is $1$, the first step has to be a divisorial contraction $f\colon X \map \PP^3$. The main case of interest for us is when $f$ is the blowup of a smooth curve.

In \cite{JPRI}, \cite{JPRII} and \cite{Cutrone-Marshburn} the two sets of authors embark on a classification of weak Fano threefolds obtained by blowups of Fano threefolds of Picard rank $1$ under different assumptions. In all three papers, the focus was the classification, while the existence of links was more of a byproduct. Moreover, the classification in the last mentioned paper was numerical in nature leaving the actual existence of some cases open.

Our approach is more closely related to the one of \cite{WeakFanos}. There, the authors give a complete
list of pairs $(g,d)$ such that the blowup of a general space curve $C$ of genus $g$ and degree $d$ produces a weak Fano threefold. This was done via geometric ways, proving also the existence of all the listed cases. Moreover, focus was also given to the links produced by the second $K_X$-negative contraction. Thus, our focus here will be the case when the blowup produces threefolds which are not weak Fano but nonetheless still fit in some Sarkisov link. In that case we will say that $C$ induces a Sarkisov link.

It is easy to show (and will also be evident in the course of this paper) that not all curves induce a Sarkisov link. In a sense, we need a way to control the $K_X$-non-negative curves. A way to do so is to reduce to the case when $C$ is contained in a surface of degree less or equal to $4$. In this case, all $K_X$-non-negative curves will be contained in the surface of low degree whose geometry we can exploit to better control their behaviour. This was the main viewpoint in \cite{WeakFanos}. In this paper we treat the case when the curve lies in a smooth cubic.

Smooth cubics in $\PP^3$ are isomorphic to $\PP^2$ blown up at $6$ points. Any curve on the blowup $\pi\colon S \map \PP^2$ of $\PP^2$ along $6$ points is linearly equivalent to $kL - \sum m_i F_i$, where $L$ is the pullback of a general line under $\pi$ and the $F_i$'s are the $\pi$-exceptional curves. In that case, we will say that the curve is of type $(k; m_1, \dots, m_6)$. Identifying the cubic with $S$, by classifying the curves that induce Sarkisov links we mean classifying their type. We may also reduce to the case where $m_i \geq m_{i+1}$ and $k\geq m_1 + m_2 + m_3$. With that in mind the main theorem of this paper is the following:
\begin{theorem}\label{theorem intro}
	Let $C \subset S \subset \PP^3$ be a curve lying on a smooth cubic surface in $\PP^3$. Suppose that the blowup of $\PP^3$ along $C$ is not weak Fano. Then $C$ induces a Sarkisov link if and only if its type (up to the assumptions $m_i \geq m_{i+1}$ and $k\geq m_1 + m_2 + m_3$) belongs to one of the following two sets:
	\begin{align*}
	\mathcal{T}^{(II)} &= \left\{ (3;\, 1,1,0,0,0,0),\, (3;\, 2,0,0,0,0,0),\, (4;\, 2,1,1,1,0,0),\, (5;\, 2,1,1,1,1,1)  \right\};\\
	\mathcal{T}^{(I)} &= \left\{ (3;\, 2,1,0,0,0,0),\, (5;\, 3,1,1,1,1,1)  \right\}.
	\end{align*}	
	Moreover, the curves in $\mathcal{T}^{(II)}$ produce Sarkisov links of Type II to terminal Fano $3$-folds of Picard rank $1$ while the curves in $\mathcal{T}^{(I)}$ produce Sarkisov links of Type I to terminal del-Pezzo fibrations of degree $5$ and $4$ respectively.
\end{theorem} 
\noindent This is proven by constructing the steps of the so-called $2$-ray game on $X$ (see section \ref{Mori dream spaces and 2-ray games}, for the definition of the $2$-ray game). 

The outline of the paper is as follows: In Section \ref{Preliminaries}, we introduce some key elements in the Sarkisov program such as the notion of Sarkisov links and that of the central model. We then define Mori dream spaces and explain their connection to the Sarkisov program via the $2$-ray game. We also prove that any curve  lying on a plane or quadric that induces a Sarkisov link already appears in the literature.

In Section \ref{Curves on cubic surfaces} we recall the basic theory of cubic surfaces. We then establish some bounds that the type of a curve must satisfy in order for it to induce a Sarkisov link and compute all possible curves that satisfy these bounds. At this point we obtain the combined list of curves appearing in Theorem \ref{theorem intro}.

In Section \ref{Existence of links} we prove that all curves in the aforementioned list actually produce Sarkisov links. This is achieved by constructing the first few steps of the $2$-ray game until we hit the weak Fano variety lying over the central model at which point the general theory assures the existence of the links. Finally, in Section \ref{Construction and study of the links} we explore  all steps of the links as well as compute some invariants of the target variety.

\textbf{Acknowledgement.} I would like to thank \emph{J\'er\'emy Blanc} for posing the problem to me as well as his guidance and invaluable help throughout. I would also like to thank \emph{Hamid Ahmadinezhad}, \emph{Pascal Fong}, \emph{St\'ephane Lamy},  \emph{Erik Paemurru},  \emph{Nikolaos Tsakanikas}, \emph{Immanuel van Santen} and \emph{Julia Schneider} for the useful discussions.

I would also like to thank the referees for their remarks and suggestions, helping greatly in the readability of the paper.

This work was supported by the Swiss National Science Foundation Grant "Birational transformations of threefolds" $200020\_178807$.

\section{Notation \& Conventions}

In this paper all varieties are assumed to be normal, projective and defined over $\CC$. A $3$-fold is a $3$-dimensional projective variety. For a variety $X$ we also define:\\
\addtolength{\tabcolsep}{-3pt} 
\begin{tabular}{rcl}
	$\WDiv(X)$ & $\defeq $ & the group of Weil divisors modulo linear equivalence;\\
	$\CDiv(X)$ & $\defeq $ & the group of Cartier divisors modulo linear equivalence;\\
	$N^1(X/Z)$ & $\defeq$ & $\{\QQ\text{-Cartier divisors modulo numerical equivalence over } Z\} \otimes \RR$;\\
	$N_1(X/Z)$ & $\defeq$ & $\{1\text{-cycles modulo numerical equivalence over } Z\} \otimes \RR$.
\end{tabular}\\
For a $\ZZ$-module $M$ we define $M_{\QQ} \defeq M\otimes_{\ZZ} \QQ$. If $z_1, z_2$ are $1$-cycles in $X$ we write $z_1 \sim_S z_2$ if there exists a surface $S$ such that $z_1$ and $z_2$ are linearly equivalent in $S$.

Finally, we denote by $\NE(X/Z)$ the cone in $N_1(X/Z)$ spanned by effective $1$-cycles and by $\Nef(X)$, $\Mov(X)$, $\Eff(X)$ and $\overline{\Eff}(X)$ the cones in $N^1(X) \defeq N^1(X/pt)$ spanned by nef, movable, effective and pseudoeffective divisors respectively.

\section{Preliminaries}\label{Preliminaries}

\subsection{The Sarkisov program}

\begin{definition}\label{singularities of pairs}
	Let $X$ be a normal variety with $K_X$ $\QQ$-Cartier and let $f \colon Y \map X$ be a resolution of singularities. Write  
	\[
	K_Y \sim_{\QQ} f^*K_X + \sum a_i E_i.
	\]
	 We say that $X$ is  \textbf{terminal} if $a_i > 0$ and \textbf{canonical} if $a_i \geq 0$.
	
	Let $\Delta = \sum d_j D_j$ be a $\QQ$-divisor on $X$, such that $K_X + \Delta$ is $\QQ$-Cartier and let $g \colon Z \map X$ be a log resolution, i.e.\ a resolution of $X$ such that $\Supp(g^{-1}(D) + \Exc(g))$ has pure codimension 1 and is simple normal crossings. Write 
	\[
	K_Y \sim_{\QQ} f^*(K_X + \Delta) + \sum a_i E_i
	\]
	where $f_*\left( \sum a_i E_i\right) = -\Delta$.	
	We say that the pair $(X, \Delta)$ is  \textbf{Kawamata log terminal} (klt for short) if $a_i > -1$ for every $i$.
\end{definition}

\begin{remark}\label{klt for any coefficient}
	Suppose that $X$ is a smooth threefold and $\Delta$ is a prime divisor. If the support of $\Delta$ is smooth, then for any $q<1$, the pair $(X, q\Delta)$ is klt.
	
	Indeed, in such case we may choose the identity as a log resolution of the pair and compare with the definition above.
\end{remark}

First we recall the definition of a Sarkisov link.
\begin{definition}[Sarkisov link]\label{Sarkisov link}
	A \textbf{Sarkisov diagram} is a commutative diagram of the form
	\[
	\xymatrix{
		X' \ar[d]_p \ar@{..>}[rr]^{\chi}  && Y' \ar[d]^q\\
		X \ar[d]_{\phi} && Y \ar[d]^{\psi} \\
		S \ar[dr]_s & & T\ar[dl]^t\\
		& R
	}
	\]
	which satisfies the following properties
	\begin{enumerate}
		\item $\phi$ and $\psi$ are Mori fibre spaces,
		\item $p$ and $q$ are divisorial contractions or isomorphisms,
		\item $s$ and $t$ are extremal contractions or isomorphisms,
		\item $\chi$ is a pseudo-isomorphism (i.e.\ an isomorphism when restricted to a subset whose complement has codimension greater than 1),
		\item all varieties of maximal dimension are $\QQ$-factorial and terminal,
		\item the relative Picard rank $\rho(Z/R)$ of any variety $Z$ in the diagram is at most 2. \label{last}
	\end{enumerate}
	Property (\ref{last}) implies that $p$ is a divisorial contraction if and only if $s$ is an isomorphism. A similar statement holds for the right hand side of the diagram. Depending whether $s$ or $t$ is an isomorphism, we get four types of Sarkisov diagrams \vspace{0.3cm}
	
	\noindent
	\begin{minipage}{.25\linewidth}
		\begin{center}
			\textbf{Type I}
		\end{center}
		\[
		\xymatrix{
			X' \ar@{..>}[r] \ar[d] & Y \ar[d]^{\psi} \\
			X  \ar[d]_{\phi} & T \ar[dl]\\
			S
		}
		\]
	\end{minipage}\hspace{-0.5cm}
	\begin{minipage}{.25\linewidth}
		\begin{center}
			\textbf{Type II}
		\end{center}
		\[
		\xymatrix{
			X' \ar@{..>}[r] \ar[d] & Y' \ar[d]\\
			X  \ar[d]_{\phi} & Y \ar[d]^{\psi} \\
			S \ar[r]^{\sim}& T
		}
		\]
	\end{minipage}\hspace{-0.5cm}
	\begin{minipage}{.25\linewidth}
		\begin{center}
			\textbf{Type III}
		\end{center}
		\[
		\xymatrix{
			X \ar@{..>}[r] \ar[d]_{\phi} & Y' \ar[d] \\
			S  \ar[rd] & Y \ar[d]^{\psi}\\
			&T
		}
		\]
	\end{minipage}\hspace{-0.2cm}
	\begin{minipage}{.25\linewidth}
		\begin{center}
			\textbf{Type IV}
		\end{center}
		\[
		\xymatrix{
			X \ar@{..>}[rr] \ar[d]_{\phi} && Y \ar[d]^{\psi} \\
			S  \ar[rd] && T \ar[ld]\\
			&R &.
		}
		\]
	\end{minipage} \vspace{0.3cm}\\
	The induced birational map $X \rmap Y$ is called a \textbf{Sarkisov link}. A diagram of the form above that satisfies all but condition $(5)$ will be called a \textbf{Sarkisov-like diagram}
\end{definition}

\begin{theorem} ({\cite[Theorem 1.1]{HM13}}) \label{Sarkisov}
	Let $X \maps{\phi} S$, $Y \maps{\psi} T$ be two Mori fibre spaces where $X$, $Y$ are $\QQ$-factorial and terminal. Then any birational map between $X$ and $Y$ can be decomposed as a sequence of Sarkisov links and isomorphisms of Mori fibre spaces.
\end{theorem}

Notice that if the starting Mori fibre space is $\PP^3 \maps{\phi} pt$, then we can only get links of Type $I$ and $II$. More specifically, the birational map $\PP^3 \rmap{}Y$ can be factored as the inverse of a divisorial contraction followed by either a pseudo-isomorphism (Type I link) or a pseudo-isomorphism and a divisorial contraction (Type II link). In the following we study the case where the first divisorial contraction is the inverse of blowing up a smooth curve. We say that such a curve $C \subset \PP^3$ \textbf{induces a Sarkisov link} if $X \defeq \Bl_C \PP^3 \map{}\PP^3$ fits into a Sarkisov diagram.

As explained in %\cite[Remark 3.5]{BCDP}
\cite[Remark 3.10]{BLZ19}, if $X_m \rmap{}Y_n$ is the pseudo-isomorphism part of a Sarkisov diagram, then its decomposition into anti-flips, flops and flips takes the following form:
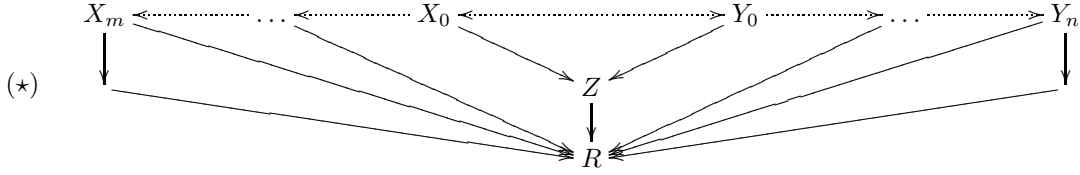
\begin{gather*}\label{Sarkisov diagram}
\begin{aligned}
\xymatrix@C=1.5cm@R=0.5cm{
	X_m\ar[d] \ar[rrrdd] & \dots \ar@{..>}[l] \ar[rrdd] & X_0 \ar[rd] \ar@{..>}[l] \ar@{<..>}[rr] & & Y_0 \ar[ld] \ar@{..>}[r] &\dots \ar@{..>}[r] \ar[lldd]&Y_n \ar[d] \ar[llldd]\\
	\ar[rrrd]&&& Z \ar[d] &&& \ar[llld]\\
	&&& R &&&
}
\end{aligned}\tag{$\star$} 
\end{gather*}
where $X_0 \smash{\xymatrix@M=1.5pt{ \ar@{<..>}[r]& }} Y_0$ is either a flop over $Z$ or $X_0 \isom Z \isom Y_0$. In the first case $X_0$ and $Y_0$ are $\QQ$-factorial and terminal weak Fanos over $R$, while in the second case $Z$ is $\QQ$-factorial and Fano over $R$. In both cases $Z/R$ is called the \textbf{central model} of the Sarkisov link/diagram.

\subsection{Mori dream spaces and 2-rays games}\label{Mori dream spaces and 2-ray games}

\begin{definition}\label{SQM, BirContr}
	Let $X$ and $Y$ be normal, projective varieties.
	
	A \textbf{small $\QQ$-factorial modification} (SQM for short) of $X$ is a pseudo-isomorphism $f\colon X \psmap X'$, where $X'$ is again normal, projective and $\QQ$-factorial.
	
	A \textbf{birational contraction} $f\colon X \rmap Y$ is a birational map such that if $(p,q)\colon W \map X \times Y$ is a resolution of $f$, then every $p$-exceptional divisor is also $q$-exceptional.
\end{definition}

\begin{definition}[Mori Dream Space]\label{MDS}
	A normal projective variety $X$ is called a \textbf{Mori Dream Space} (MDS for short) if it satisfies the following:
	\begin{enumerate}
		\item $X$ is $\QQ$-factorial and $\Pic(X)_{\QQ} = \N^1(X)_{\QQ}$,
		\item $\Nef(X)$ is generated by finitely many semi-ample divisors and
		\item there are finitely many SQMs $f_i\colon X \psmap{}X_i$ such that each $X_i$ satisfies $(1)$ and $(2)$ and $\Mov(X)$ is the union of $f_i^*\Nef(X_i)$.
	\end{enumerate}
\end{definition}

\begin{proposition}[{\cite[Proposition 1.11]{Hu-Keel}}]\label{Properties of MDS}	
	Let $X$ be an MDS. Then the following hold.
	\begin{enumerate}
		\item The Minimal Model Program (MMP for short) can be carried out for any divisor $D$ on $X$. That is, for any $D$-negative extremal ray of $\NE(X)$ the necessary contractions and flips exist, any sequence of flips terminates, and if at some point $D$ becomes nef then at that point it becomes semi-ample.
		\item 	The $f_i$'s in property $(3)$ of Definition \ref{MDS} are the only SQMs of $X$. 		
		Moreover, there are finitely many birational contractions $g_i \colon X \rmap{}Y_i$, such that 
		\[
		\overline{\Eff}(X)= \bigcup_i \mathcal{C}_i,
		\]
		where $\Eff(X)$ denotes the cone of effective divisors in $\N^1(X)$ and
		\[
		\mathcal{C}_i = g_i^*\Nef(Y_i) + \RR_{\geq 0}\{E_1, \dots, E_k\},
		\]
		with $E_1, \dots, E_k$ being the prime divisors contracted by $g_i$. The $\mathcal{C}_i$'s are called the \textbf{Mori chambers} of $X$.
		\item Adjacent Mori chambers are related by a $D$-flip for some $D \in \Div(X)$.
	\end{enumerate}
\end{proposition}

%\begin{remark}
%	Using the exponential sequence, one sees that property $(1)$ is equivalent to the property $h^1(\OO_X) = 0$.
%\end{remark}

\begin{definition}
	A normal, projective and $\QQ$-factorial variety $X$ is called
	\begin{itemize}
		\item \textbf{weak Fano} if the anti-canonical divisor $-K_X$ is nef and big;
		\item \textbf{of Fano type} if there is an effective $\QQ$-divisor $\Delta$  such that the pair $(X,\Delta)$ is klt and $-(K_X + \Delta)$ is ample.
	\end{itemize}	
\end{definition}

\begin{lemma}\label{weak Fano implies Fano type}
	If $X$ is terminal and weak Fano, then $X$ is of Fano type.
\end{lemma}

\begin{proof}
	Since $X$ is weak Fano, by definition $-K_X$ is big. Thus is can be written as
	\[
	-K_X \sim A + E
	\]
	where $A$ and $E$ are ample and effective $\QQ$-divisors respectively (see \cite[Corollary 2.2.7]{Lazarsfeld1}). For any $k > 1$ we write
	\[
	k(-K_X) = (k - 1)(-K_X) + (-K_X) \sim (k-1)(-K_X) + A + E.
	\]
	Since $-K_X$ is nef, $A' \defeq (k-1)(-K_X) + A$ is ample and so $-(K_X + \frac{1}{k}E)$ is ample. Moreover we can choose $k$ sufficiently large such that the pair $(X,\frac{1}{k}E)$ is klt.
\end{proof}

\begin{proposition}[{\cite[Corollary 1.3.2]{BCHM}}]\label{Fano type is MDS}
	If $X$ is of Fano type then $X$ is a Mori Dream Space.
\end{proposition}

Property \ref{Properties of MDS} (1) is a fundamental property of MDSs and allows us to play 2-ray games and possibly construct Sarkisov links.

\begin{proposition}\label{Induces link iff MDS}
	Let $C \subset \PP^3$ be a smooth curve and $X \defeq \Bl_C \PP^3 \map \PP^3$ the blowup of $\PP^3$ along $C$. Then $C$ induces a Sarkisov link if and only if $X$ is an MDS and for every birational contraction $f\colon X \rmap{}Y$ (see Definition \ref{SQM, BirContr}) such that $Y$ is $\QQ$-factorial, $Y$ is also terminal.
\end{proposition}

\begin{proof}
	Suppose that $C$ induces a Sarkisov link and let $X_0$, $Y_0$ be the varieties over the central model $Z/pt$ of the Sarkisov link
	\[
	\xymatrix@R=0.5cm{
		X_0 \ar[rd] \ar@{<..>}[rr] & & Y_0 \ar[ld] \\
		& Z \ar[d]\\
		& pt
	}
	\]
	so that $X_0 \psmap{}Y_0$ is a flop or an isomorphism.
	Then $X_0$ is a terminal weak Fano threefold, hence by Lemma \ref{weak Fano implies Fano type} of Fano type and in turn an MDS. By Proposition \ref{Properties of MDS} (2) as well as property $(3)$ of Definition \ref{MDS}, $X$ is also an MDS. 
	
	Let $g_1$, $g_2$ denote the birational contractions with targets the Mori fibre spaces $\PP^3/pt$ and $Y/T$. We claim that $\mathcal{F}_1 \defeq g_1^*\Nef(\PP^3)$ and $\mathcal{F}_2 \defeq g_2^*\Nef(Y)$ are extremal in the movable cone. Indeed, let $D$ be a divisor in $\mathcal{F}_i$ and denote by $E_i$ the exceptional divisor of $g_i$. Then $E_i$ is covered by infinitely many $g_i$-exceptional curves which are $E_i$-negative. For any $\kappa > 0$, $\kappa D + E_i$ is negative against all those curves and thus must contain them all and by extension $E_i$. Thus $\kappa D + E_i$ is not movable, proving the claim.	
	 Moreover, since $\mathcal{F}_1$ and $\mathcal{F}_2$ are 1-dimensional and $\rho(X) = 2$ they generate the movable cone of $X$. Thus, by (2) and (3) of Proposition \ref{Properties of MDS}, the Sarkisov link factors through all the SQMs $f_i\colon X \psmap{}X_i$ 
	 \footnote{uniquely in fact, since $N^1(X)$ is $2$-dimensional and so any path in the between $\mathcal{F}_1$ and $\mathcal{F}_2$ hits all other chambers in a unique order}
	 thus all of them appear in the Sarkisov diagram making all $X_i$ terminal. Moreover, the only birational contractions of $X$ with $\QQ$-factorial targets are $X \map{}\PP^3$ and $X \rmap{}Y$ both of which are terminal.
	
	Conversely, if $X$ is an MDS then we can run the so-called 2-ray game. This is explained in detail in \cite[Section 2F]{BLZ19}; here we give a rough idea. Choose an ample divisor $A$ on $X$ and run the $(-A)$-MMP. Any such MMP must terminate with a Mori fibre space. Since $\rho(X) = 2$  on the first step of the $(-A)$-MMP we have a choice between two $(-A)$-negative rays to contract. One MMP outputs $\PP^3/pt$, while the other outputs another Mori fibre space $Y/T$. 
	Now using the fact that the targets of all SQMs are terminal one can check that the diagram produced by this process satisfies all the properties of Definition \ref{Sarkisov link}.
\end{proof}

\subsection{Notation \& Setup}\label{setup and notation}

\begin{notation}\label{notation}
	Throughout the rest of the paper and unless otherwise stated, $C$ will be a smooth space curve and $\pi\colon X \map \PP^3$ will be the blow up of $\PP^3$ along $C$. We will denote by $f$ the numerical class of a fibre of $\pi$ over a point of $C$ and by $l$ the numerical class of the pullback of a general line in $\PP^3$. 
	
	Dually in $\N^1(X)$, we will denote by $E$ the class of the $\pi$-exceptional divisor and by $H$ the class of the pullback of a general hyperplane in $\PP^3$. These classes generate their respective vector spaces and the intersection matrix is determined by the relations: $H \cdot l = 1$, $H\cdot f = E\cdot l = 0$ and $E\cdot f = -1$.	
	Note that in this notation we have the relation: $K_X \sim -4H + E$.
\end{notation}

The Mori cone $\overline{\NE}(X)$ of $X$ is a two dimensional cone, with one extremal ray generated by $f$ which is $K_X$-negative. Thus, whether or not $X$ is a weak Fano is determined by the sign of the second generating ray relative to the canonical divisor.

By abuse of notation, we will also denote by $H$ and $l$ the classes of a hyperplane and a line in $\N^1(\PP^3)$ and $\N_1(\PP^3)$ respectively.

\begin{lemma}\label{curves}
	The Mori cone $\overline{\NE}(X)$ is spanned by two extremal rays; the first generated by $f$ and the second by the class $l-rf$ with $r \in \RR$ maximal among the pseudo-effective classes.
	
	Dually, $\overline{\Eff}(X)$ is spanned by the two extremal rays; the first generated by $E$ and the second by the class $H -  rE$ with $r\in \RR$ maximal among the pseudo-effective classes.
\end{lemma}

\begin{proof}
	It is clear that $f$ generates an extremal ray since it is the fibre of a contraction. Let $l - sf$ be a pseudo-effective class, with $s \leq r$, then clearly
	\[
	l - sf = l-rf  + (r-s)f.
	\]
	The proof is similar for the dual statement.
\end{proof}

\begin{remark}\label{class}
	The effective representatives of an effective class $dH - mE$, with $d,m \in \NN$, are strict transforms of surfaces of degree $d$ having multiplicity $m$ along $C$.	
	
	Similarly, the effective representatives of an effective class $dl - mf$ which do not lie on $E$ are strict transforms of curves of degree $d$ meeting $C$ at $m$ points counted with multiplicities, that is, the scheme theoretic intersection of $C$ with such a representative is a $0$-dimensional scheme of length $m$.
	
	We will call such a curve an $\boldsymbol{m}$\textbf{-secant curve} to $C$ and when $d = 1$, an $\boldsymbol{m}$\textbf{-secant line}, when $d=2$ an $\boldsymbol{m}$\textbf{-secant conic} and so on.
\end{remark}

We note that if $C$ induces a Sarkisov link then by Proposition \ref{Induces link iff MDS} $X$ must be an MDS and thus its cone of curves is closed and rationally generated. In that case the second extremal ray is generated, over $\QQ$, by a class of the form $dl-mf$ with $\frac{m}{d}$ maximal among all effective classes.

\subsection{Curves in planes or quadrics}
We now show that if $C$ lies on a plane or quadric, then it induces a Sarkisov link if and only if the blowup $\Bl_C\PP^3$ is Fano. All such curves have been classified in \cite{WeakFanos} and so any open cases cannot lie on surfaces of degree $1$ or $2$. But first we prove a result which is essential for the rest of the paper.

\begin{proposition}\label{infinte K-positive curves}
	Let $\chi\colon Y \psmap{} Y'$ be an anti-flip between $\QQ$-factorial terminal 3-folds, $z$ be a curve in $Y$ which is not in the indeterminacy locus of $\chi$ and $z'$ be its strict transform under $\chi$. Then 
	\[
	K_Y\cdot z \leq K_{Y'}\cdot z'.
	\]
	In particular, if $X$ is the blowup of a Fano threefold of Picard rank 1 such that it admits an infinite number of $K_X$-non-negative curves then $X$ does not fit into a Sarkisov link.
\end{proposition}

\begin{proof}
	Let
	\[
	\xymatrix{
		& W \ar[ld]_p \ar[rd]^q\\
		Y \ar@{..>}[rr]^{\chi}&& Y'
	}
	\]
	be a resolution of indeterminacies of $\chi$. Writing the ramification formulas for $p$  and $q$ and comparing them we get
	\begin{gather}\label{ram}
	q^*K_{Y'} = p^*K_{X_{Y}} + (E_p - E_q),
	\end{gather}
	where $E_p$ and $E_q$ are sums of $p$ and $q$-exceptional divisors, respectively. Since $\chi$ is a pseudo-isomorphism, $\Supp E_p = \Supp E_q$. Moreover, the discrepancies decrease under an anti-flip (see \cite[Lemma 9.1.3]{Matsuki}) and so $E_p - E_q \geq 0$.	
	Let $z_W \subset W$ be the curve dominating both $z$ and $z'$. Then
	\[
	K_{Y'}\cdot z' = q^*K_{Y'}\cdot z_W = p^*K_{Y}\cdot z_W + (E_p-E_q)\cdot z_W \geq p^*K_{Y}\cdot z_W = K_{Y}\cdot z.
	\]
	
	Suppose now that $X$ is as described in the statement. Assuming $X$ fits into a Sarkisov diagram we will derive a contradiction. $\overline\NE(X)$ has two extremal rays, one of which, which we will denote by $R$, is $K_X$-non-negative and the other corresponding to the blowup morphism is $K_X$-negative. Since $X$ fits into a Sarkisov diagram, $R$ must also be contractible. We distinguish two cases. 
	
	If $-K_X$ is nef then all the $K_X$-non-negative curves are actually $K_X$-trivial curves. In that case $R$ has to contain all of the infinitely many $K_X$-trivial curves thus the contraction $f\colon X \map Y$ of $R$ is not small. If it is not divisorial then the induced diagram cannot be a Sarkisov link. Finally, if $f$ is divisorial, writing the ramification formula for $f$ and intersecting both sides with a contracted curve we get that $X$ has canonical but not terminal singularities, giving us a contradiction to the existence of the Sarkisov diagram.
	
	On the other hand if $-K_X$ is not nef, playing the $2$-ray game on $X$ and after a finite number of $K_X$-positive steps, i.e.\ anti-flips, we must arrive to the variety $X_0$ lying over the central model of the link which is either Fano or weak-Fano. By the first part of the proposition, $X_0$ must still contain an infinite number of $K_{X_0}$-non-negative curves. If $X_0$ is Fano this is already a contradiction. If $X_0$ is weak-Fano then the contraction of the $K_{X_0}$-trivial ray of $\overline\NE(X_0)$ must be small, i.e.\ contain a finite number of curves, again a contradiction.
\end{proof}

The following statement follows implicitly from \cite{WeakFanos}.

\begin{proposition}\label{non-existence for planes and quadrics}
	Let $C$ be a smooth curve contained in a plane or quadric. If $C$ induces a Sarkisov link, then $X$ is weak Fano.
\end{proposition}

\begin{proof}
	If $C$ is contained in a plane or quadric, in \cite[Proposition 3.1]{WeakFanos} the authors prove that if the blowup of $X$ is not weak Fano, then $C$ admits an infinite number of $m$-secants lines, with $m\geq 5$. By Remark \ref{class}, their class in $N_1(X)$ is $l-mf$ and thus correspond to $K_X$-positive curves. By Proposition \ref{infinte K-positive curves}, we conclude that $C$ does not induce a Sarkisov link.
\end{proof}

\begin{remark}\label{BL lists are exhaustive}
	We note that, by Proposition \ref{non-existence for planes and quadrics}, the lists of \cite{WeakFanos} are exhaustive in the cases where the curve is contained in a plane or quadric. That is, every curve that lies on plane or quadric whose blowup induces a Sarkisov link is studied and appear in their lists.
\end{remark}

\section{Curves on smooth cubic surfaces}\label{Curves on cubic surfaces}

The goal of this section is to give necessary conditions for curves lying on smooth cubic surfaces to induce a Sarkisov link.

If $S$ is the blowup of $\PP^2$ along $6$ points $p_1,\dots, p_6$ in general position (i.e. no $3$ on a line and no $6$ on a conic), then $-K_S$ is ample and the anti-canonical system gives an embedding in $\PP^3$ where the image is a smooth cubic surface. Conversely, any smooth cubic surface in $\PP^3$ is obtained like that (see \cite[Section 4]{WeakFanos}). In the following we will denote the image of $S$ under the anti-canonical system again by $S$ and when there is no confusion we will not distinguish between the two.

We recall some basic facts about curves on $S$. Denote by $L$ the numerical class of the pullback of a general line under $S \map \PP^2$ and by $F_i$ the class of the fibre over $p_i$. These classes generate $\N^1(S)$. The intersection form is given by the diagonal $(1,-1,\dots,-1)$. We will say that a 1-cycle $z \in \N_1(S)$ is of \textbf{type} $(k;\,m_1, m_2, m_3, m_4, m_5, m_6)$ if its class is numerically equivalent to $kL -\sum m_1F_1$. The canonical class on $S$ is then of type $(-3;\, 1,\dots ,1)$.  Moreover the cone of curves $\NE(S)$ is closed and generated by the classes of the $(-1)$-curves. There are 27 such classes corresponding to the 6 exceptional curves, the strict transforms of the 15 lines through two of the points or the $6$ conics through five of the points. Their classes are respectively $F_i$, $L- F_i - F_j$ and $2L - F_1 -\dots - F_6 + F_i$. Their images are lines in $\PP^3$ which we will denote by $e_i$, $l_{i,j}$ and $c_i$ respectively.

We stick to the notation introduced in Notation \ref{notation}. For the following statements we will also always assume that $C$ lies on a smooth cubic surface $S$.

\begin{proposition}\label{lines are enough}
	Let $\gamma$ be a curve in $X$ such that $\gamma \sim dl-mf$ with $\frac{m}{d} > 3$. Then $C$ admits an $n$-secant line with $n \geq \frac{m}{d}$. In particular, $X$ is weak Fano if and only if $C$ admits no $m$-secant line with $m \geq 5$.
\end{proposition}

\begin{proof}
	We write $T$ for the strict transform of $S$. Then $T$ and $S$ are isomorphic and
	\[
	T \cdot \gamma = (3H - E)(dl-mf) = 3d - m < 0
	\]
	thus $\gamma$ is contained in $T$. The cone of curves of $T$ is generated by the strict transforms of the 27 lines $l_i$ in $S \subset \PP^3$ and so we can write
	\[
	\gamma \sim_{\scriptscriptstyle T} l_1 + \dots + l_k \implies \gamma \equiv_{\scriptscriptstyle{X}} l_1 + \dots + l_k.
	\]
	Intersecting with the restriction of a general hyperplane of $\PP^3$ we get that $k = d$. Intersecting with $E$ we get
	\[
	E(l_1 + \dots + l_d) = m.
	\]
	Therefore, at least one of the lines $l_i$ intersects $E$ at $n$ points counted with multiplicity, with $n \geq \frac{m}{d}$.
\end{proof}

By the previous lemma we see that the property of $X$ not being weak Fano is equivalent to the existence of at least one $m$-secant line with $m>4$. On the other hand, the following, we prove that $C$ must not admit ``too many'' in a sense such lines.

\begin{lemma}\label{conic bundle}
	Let $l_1$, $l_2$ be two distinct, intersecting lines on a smooth cubic surface $S$. Then there exists a pencil of conics $\mathcal{C}$ on $S$ such that each element $c \in \mathcal{C}$ is linearly equivalent to $l_1 + l_2$.
\end{lemma}

\begin{proof}
	Denote by $P_1$ the unique plane containing both $l_1$ and $l_2$, and let $l_3$ be the residual line of the intersection of $P_1$ with $S$. Then the pencil of planes containing $l_3$ gives us a residual pencil of conics on $S$ which are linearly equivalent to $l_1 + l_2$.
\end{proof}

\begin{lemma}\label{term}
	Suppose that $C$ admits two distinct, $m_1$ and $m_2$-secant lines $l_1$ and $l_2$ which intersect each other. If $m_1+m_2 \geq 8$, then $C$ does not induce a Sarkisov link.
\end{lemma}

\begin{proof}	
	By Lemma \ref{conic bundle} there is a pencil of conics such that each element intersects $C$ at $m_1 + m_2$ points counted with multiplicity. Their strict transforms on $X$ give us an infinite family of $K_X$-non-negative curves. We conclude by Corollary \ref{infinte K-positive curves}.
\end{proof}

We now set-up an algorithm to compute all the curves on smooth cubics, candidates to induce a Sarkisov link. We note that the conditions we will impose at this stage are necessary but not sufficient to guarantee that the curve induces a Sarkisov link.

\begin{remark}\label{assumptions on the type}
	Suppose that $C$ is of type $(k;m_1,m_2,m_3,m_4,m_5,m_6)$. Up to reordering, we may assume 
	\[
	m_1 \geq m_2 \geq \dots \geq m_6.
	\]
	Moreover, as explained in \cite[Set-up  4.1]{WeakFanos}, by performing a number of quadratic transformations
	\[
	\xymatrix{
			&S \ar[rd] \ar[ld]\\
	\PP^2 \ar@{-->}[rr] && \PP^2 		
	}
	\]
	based on 3 of the 6 points and changing the left hand side morphism to the right hand side one we may assume that 
	\[
	k \geq m_1 + m_2 + m_3.
	\]
	In what follows, we will assume that the type of $C$ satisfies these conditions.
\end{remark}

\begin{lemma}\label{secancy}
	Let $C \subset S$ be a curve of type $(k;m_1,m_2,m_3,m_4,m_5,m_6)$. Then we have the following inequalities for the intersection of $C$ with the various lines on $S$:
	\begin{align*}
	C \cdot c_1 \geq C \cdot L, && C \cdot l_{5,6} \geq C\cdot l_{i,j}, && C \cdot l_{k,6} \geq C\cdot l_{k,j}, && C \cdot e_1 \geq C \cdot e_i,
	\end{align*}
	where $L$ is any line on $S$.
\end{lemma}

\begin{proof}
	We have
	\[
	C \cdot c_1 = 2k - (m_2 + \dots + m_6) \geq 2k - (m_1 + \dots + \hat m_i + \dots + m_6) = C\cdot c_i,
	\]
	where the hat notation means that the corresponding term does not appear in the expression. In a similar manner we get the rest of the inequalities.	
\end{proof}

Using these inequalities as well as Lemma \ref{term} we will bound the quantities $k, m_1, \dots, m_6$.

\begin{proposition}\label{necessary conditions}
	Let $C \subset S$ be a curve of type $(k;m_1,m_2,m_3,m_4,m_5,m_6)$ and suppose that $C$ induces a Sarkisov link and that $X \defeq \Bl_{C}\PP^3$ is not weak Fano. Then
	\[
	k \leq 9 \qquad m_1 \leq 8  \quad \text{and} \quad m_2,\dots, m_6 \leq 2.
	\]
\end{proposition}

\begin{proof}
	Since we assume that $X$ is not weak Fano and by Proposition \ref{lines are enough}, $C$ admits an $m$-secant line with $m \geq 5$. Thus, by Lemma \ref{secancy} we have $C \cdot c_1 \geq 5$.
	
	For $C$ to induce a Sarkisov link, by Lemma \ref{term} we need to make sure that the sum of any two intersecting lines has intersection less than $8$ with $C$. Applying this to the lines intersecting $c_1$ we get 
	\[
	C\cdot(c_1 + l_{1,j}) \leq 7 \quad \text{ and } \quad C\cdot(c_1 + e_n) \leq 7
	\]
	for $n \neq 1$. The second inequality above gives $C\cdot e_n \leq 2$, hence
	\[
	m_2, \dots, m_6 \leq 2.
	\]
	Now the inequality $C\cdot(c_1 + l_{1,j})$ for $j=6$ together with the bounds on the $m_2, \dots, m_6$ give
	\[\label{star}\tag{$\star$}
	3k - m_1 - \dots - m_5 - 2m_6 \leq 7 \ra 3k \leq 19+ m_1.
	\]
	For $k = 1$, the curve $C$ cannot have any $m$-secants with $m \geq 5$. Since we assume that $X$ is not weak Fano, we get $k\neq 1$ and so by B\'ezout's theorem on $\PP^2$ we get $k > m_1$. Combining with (\ref{star}) we get the last bounds $k \leq 9$ and $m_1 \leq 8$.
\end{proof}

Using the bounds and checks introduced above we can set up an algorithm to compute all the candidate curves. They are presented in the table bellow.

\begin{table}[H]	
	\resizebox{1 \textwidth}{!}{	
		\rowcolors{3}{gray!25}{white}
		$
		\begin{array}{| c | c | c | c | c | c |}
		\hline
		\rowcolor{gray!50} & & & & &\\[-10pt]
		\rowcolor{gray!50}
		\# &  \text{ Type }  &(C\cdot c_1, \dots ,C\cdot c_6) &  (C\cdot l_{1,2},\dots, C\cdot l_{5,6}) & \deg(C) & \g(C) \\[2pt]
		\hline
		1 & (3;1, 1, 0, 0, 0, 0) & (5, 5, 4, 4, 4, 4)  &  (1, 2, 2, 2, 2, 2, 2, 2, 2, 3, 3, 3, 3, 3, 3) & 7& 1\\ 
		2 & (3;2, 0, 0, 0, 0, 0) &  (6, 4, 4, 4, 4, 4)  &  (1, 1, 1, 1, 1, 3, 3, 3, 3, 3, 3, 3, 3, 3, 3) & 7 & 0\\ 
		3 & (4;2, 1, 1, 1, 0, 0) &  (5, 4, 4, 4, 3, 3)  &  (1, 1, 1, 2, 2, 2, 2, 3, 3, 2, 3, 3, 3, 3, 4) & 7 & 2\\ 
		4 & (5;2, 1, 1, 1, 1, 1) &  (5, 4, 4, 4, 4, 4)  &  (2, 2, 2, 2, 2, 3, 3, 3, 3, 3, 3, 3, 3, 3, 3) & 8 & 5\\
		\hline
		5 & (3;2, 1, 0, 0, 0, 0) &  (5, 4, 3, 3, 3, 3)  &  (0, 1, 1, 1, 1, 2, 2, 2, 2, 3, 3, 3, 3, 3, 3) & 6 & 0\\ 
		6 & (5;3, 1, 1, 1, 1, 1) &  (5, 3, 3, 3, 3, 3)  &  (1, 1, 1, 1, 1, 3, 3, 3, 3, 3, 3, 3, 3, 3, 3) & 7 & 3\\ 
		\hline
		\end{array}
		$
	}
	\caption{List of candidate curves}		
	\label{List of candidate curves}
\end{table}

\begin{remark}
	The curves above are exactly the ones of Theorem \ref{theorem intro}.
\end{remark}

\section{Existence of the links}\label{Existence of links}

In the previous section we produced a table of types of curves lying on some smooth cubic, which satisfy the necessary criterion set by Proposition \ref{infinte K-positive curves} to induce a Sarkisov link. In this section we will prove that they actually do induce Sarkisov links. 

%In this section we set out to prove the existence of the links for all the candidate curves of Table \ref{List of candidate curves}. This is done in several steps which we now describe.
%
%First, we prove that the blow up of $\PP^3$ along the curve is an MDS. This allows to play $2$-ray games and construct Sarkisov-like diagrams, i.e.\ diagrams which satisfy all the properties of Definition \ref{Sarkisov link} apart from the assumption on the terminal singularities.
%
%We then construct local pseudo-isomorphisms centred at $l_m$ from the extremal curve germs of $(X,l_m)$ where $l_m$ is an $m$-secant line and prove that the targets have again terminal singularities. Here, by an extremal curve germs of $(X,l_m)$ we mean an analytic germ of a terminal threefold $X$ such that $l_m$ is contractible and does not contain any singular points of $X$.
%
%Finally, if $3 < m_1 < \dots < m_k$ are integers such that $C$ admits $m_i$-secant lines for any $i \in \{1,\dots, k\}$, we prove that the first $k$ pseudo-isomorphisms of the Sarkisov-like diagram are the local pseudo-isomorphisms constructed in the previous step. More precisely, if $X \psmap X_k$ is the pseudo-isomorphism associated to the extremal ray generated by $m_k$-secant lines, then the strict transform of any $m_{k-1}$-secant line generates an extremal ray of $X_k$.
%
Again, unless otherwise stated, we stick to the notation introduced in Notation \ref{notation}.

\subsection{Some properties of the surfaces \texorpdfstring{$\FF_n$}{FFn}}

For a proof of the facts that follow see \cite[Section IV]{Beauville}.

We will denote by $\FF_n$ the $\PP^1$-bundle $p\colon \PP(\OO_{\PP^1} \oplus \OO_{\PP^1}(n)) \map \PP^1$. Then $\FF_n$ admits a unique section $\sigma$ with self intersection $-n$. We will call any section of self-intersection $n$ an \textbf{$n$-section}. If we denote by $f$ a fibre of $p$ then the Mori cone $\overline{\NE}(\FF_n)$ is spanned by the classes of $\sigma$ and $f$. The intersection matrix is 
\[
\begin{pmatrix*}[c]
-n & 1 \\
\phantom{-}1 & 0
\end{pmatrix*},
\]
with respect to the basis $\{\sigma,f\}$. The canonical class is $K_{\FF_n} = -2\sigma - (n+2)f$. Finally, $\FF_n$ is smooth and rational, so $h^i(\FF_n, \OO_{\FF_n}) = h^i(\PP^2,\OO_{\PP^2}) = 0$, for $i \geq 1$. Thus the Riemann-Roch theorem takes the form
\[
\chi(\OO_{\FF_n}(D)) = \frac{D(D-K_{\FF_n})}{2} + 1,
\]
for any $D \in \Div(\FF_n)$.
In particular, if $D \sim a \sigma + b f$, we have
\[
\chi(D) = (a+1)(b+1) - n\frac{a(a+1)}{2}.
\]

We now prove a lemma on the cohomology of divisors on $\FF_n$.

\begin{lemma}\label{cohomology Fn}
	Let $D$  be a divisor on $\FF_n$. Then the following hold:
	\begin{enumerate}
		\item if $D\cdot f = -1$, then $h^i(\FF_n,D) = 0$ for every $i \geq 0$;
		\item if $D\cdot f \geq 0$ and $D \cdot \sigma \geq -1$, then $h^1(\FF_n,D) = 0$.
	\end{enumerate}
\end{lemma}

\begin{proof}
	The proof can be found in \cite{Coskun-Huizenga}, we merely replicate it for the convenience of the reader.
	
	Assume first that $D\cdot f = -1$. Then $D$ is linearly equivalent to $-\sigma + kf$ for some $k\in \ZZ$. Since $\overline{\NE}(\FF_n)$ is spanned by the classes of $\sigma$ and $f$ we get that $h^0(\FF_n,-\sigma + kf) = 0$ and by Serre duality, $h^2(\FF_n,-\sigma + kf) = h^0(\FF_n, -\sigma -(n+2+k)f) =0$. Finally using Riemann-Roch one can compute $\chi(\FF_n, D) = 0$ and so $h^1(\FF_n,-\sigma + kf) = 0$. Hence, we obtain $(1)$.
	
	Now assume that $D\cdot f = k \geq 0$ and $D \cdot \sigma \geq -1$. Define $N := D-(k+1)\sigma$. We will show by induction on $i$ that $h^1(\FF_n,N+i\sigma) = h^1(\FF_n,D-(k+1-i)\sigma) = 0$ for every $0 \leq i \leq k+1$. For $i = 0$ we have $N\cdot f = -1$ and thus we are done by $(1)$. 
	
	For the inductive step we assume that $h^1(\FF_n,N+(i-1)\sigma) = 0$ and consider the short exact sequence
	\[
	0 \map \OO_{\FF_n}\left(N + (i-1)\sigma\right) \map \OO_{\FF_n}(N+i\sigma) \map  \OO_{\sigma}(N+i\sigma) \map 0.
	\]
	Then the long exact sequence in cohomology yields
	\[
	H^1(\FF_n,\OO(N + (i-1)\sigma)) \map H^1(\FF_n,\OO(N + i\sigma)) \map H^1(\sigma,\OO(N + i\sigma)) \isom H^1(\PP^1,\OO_{\PP^1}(\kappa)) = 0
	\]
	where the last equality follows from the fact that $\kappa = D\cdot \sigma + n(k+1-i) \geq -1$. Thus using the inductive hypothesis we get $h^1(\FF_n,N - i\sigma)=0$. 
	
	For $i = k+1$ we get $h^1(\FF_n,N + (k+1)\sigma) = h^1(\FF_n,D) = 0$
	This is $(2)$.
\end{proof}

In the following, $\overline{\FF}_n$ will denote the surface obtained by contracting the $(-n)$-section of $\FF_n$.
To conclude the subsection, we prove a lemma about the intersection theory on $\overline{\FF}_n$.

\begin{lemma}\label{FF_n bar}\sloppy{
		Let $\phi\colon\FF_n \map \overline{\FF}_n$ be the contraction of the $(-n)$-section on $\FF_n$. Then $\WDiv(\overline{\FF}_n) = \langle \bar{f} \rangle$ and $\CDiv(\overline{\FF}_n) = \langle \bar{\sigma}_+ \rangle$ with $\bar{\sigma}_+ = n\bar{f}$, where $\bar{\sigma}$ and $\bar{f}$ denote the images of an $n$-section and a fibre respectively. We also have $\bar{\sigma}_+^2 = n$ which implies that $\bar{f}^2 = \frac{1}{n}$. }
	%Finally we have $K_{\overline{\FF}_n} = $.
\end{lemma}

\begin{proof}
	The morphism $\phi$ is given by the linear system $|\sigma + nf|$ (i.e.\ the linear system of $n$-sections). The rank of $\CDiv(\overline{\FF}_n)$ is 1 and a general hyperplane section is the image of an $n$-section, thus $\CDiv(\overline{\FF}_n) = \langle \bar{\sigma} \rangle$. Moreover $(\sigma + nf)f = 1$ thus the image of a fibre is a line in $\PP^N$ and so $\WDiv(\overline{\FF}_n) = \langle \bar{f} \rangle$. 
	
	The hyperplanes sections passing through the singular point $\phi(\sigma)$ are images of members of $|\sigma + nf|$ which have $\sigma$ as an irreducible component. Since $\sigma$ is contracted, hyperplane sections through $\phi(\sigma)$ are equivalent to $n\bar{f}$ which implies that $\bar{\sigma}_+ = n\bar{f}$.
	
	Finally, since an $n$-section does not meet the exceptional divisor of $\phi$ we have $\bar{\sigma}_+^2 = (\sigma + nf)^2 = n$ from which we also conclude that $\bar{f}^2 = \left(\frac{1}{n}\bar{\sigma}_+\right)\left(\frac{1}{n}\bar{\sigma}_+\right) = \frac{1}{n}$.
\end{proof}

\subsection{\texorpdfstring{$X$}{X} is a Mori Dream Space}

\begin{proposition}\label{X is an MDS}
	Let $C\subset \PP^3$ be a smooth curve lying on a smooth cubic surface $S \subset \PP^3$. Then $X \defeq \Bl_C\PP^3$ is an MDS.
\end{proposition}

\begin{proof}	
	If $X$ is weak Fano, then by Lemma \ref{weak Fano implies Fano type} and Proposition \ref{Fano type is MDS}, it is an MDS.
	
	Suppose $X$ is not weak Fano. By Lemma \ref{curves}, $\overline{\NE}(X)$ is generated by $f$ and $dl - mf$ with $\frac{m}{d} > 4$ and then by Proposition \ref{lines are enough}, $d=1$ and $m\geq 5$. Moreover, $K_X \sim -4H + E$ and $S \sim 3H - E$, where by abuse of notation $S$ also denotes the strict transform of $S$ in $X$. The intersections among those classes are given by the table
	\[
	\begin{array}{c|c|c}
	& f 	& l - mf\\\hline
	S	 & 1  &3 - m\\\hline
	K_X& -1 & -4+m 
	\end{array}
	\]
	Then by Remark \ref{klt for any coefficient}, for any rational number $0<q<1$ the pair $(X,qS)$ is klt. If moreover $1 -  \frac{1}{m-3} < q < 1$, then intersecting $-(K_X + q S)$ with both $f$ and $l-mf$ we get strictly positive numbers. Kleiman's criterion for ampleness (see \cite[Theorem 1-2-5]{Matsuki} for a statement and \cite{Kleiman} for a proof) implies that $-(K_X + q S)$ is ample and so $(X,qS)$ is log Fano. Thus by definition, $X$ is of Fano type and by Proposition \ref{Fano type is MDS}, $X$ is an MDS.
\end{proof}

\subsection{Construction of the pseudo-isomorphisms}

\begin{proposition}\label{normal bundles of lines}
	Let $l$ be the strict transform of an $m$-secant line under $X \map \PP^3$. If $m \geq 3$, then 
	\[
	N_{l/X} \isom \OO_{\PP^1}(-1) \oplus \OO_{\PP^1}(3-m),
	\]
	where $N_{l/X}$ denotes the normal bundle of $l$ in $X$.
\end{proposition}

\begin{proof}
	Consider the short exact sequence
	\[
	0 \map N_{l/S} \map N_{l/X} \map (N_{S/X})|_l \map 0
	\]
	obtained by dualizing the conormal exact sequence (see \cite[p.\ 79]{Fulton-Lang}). 
	We have 
	\[
	\deg (N_{S/X})|_l = S|_S\cdot l = (3H - E)|_S \cdot l = (-3K_S - C)\cdot l = 3-m.
	\]
	Moreover $\deg(N_{l/S})= -1$ because $l$ is a $(-1)$-curve in $S$.
	Since $l$ is a rational curve, all line bundles on it are of the form $\OO(d)$, where $d$ is the degree of the bundle and so the exact sequence above becomes
	\[\label{dagger}\tag{$\dagger$}
	0 \map \OO_{\PP^1}(-1) \map N_{l/X} \map \OO_{\PP^1}(3-m) \map 0.
	\]	
	Extensions of lines bundles
	\[
	0 \map L \map M \map N \map 0
	\]
	are classified by $H^1(X,N^{-1}\otimes L)$ (see \cite[p.\ 31]{Friedman}) and in the case of (\ref{dagger}) by $H^1(\PP^1, \OO_{\PP^1}(m-4))$.
	However,
	\[
	h^1(\PP^1, \OO_{\PP^1}(m-4)) = h^0(\PP^1, \OO_{\PP^1}(2-m)) = 0
	\]
	since $m \geq 3$. Thus (\ref{dagger}) is the unique extension and is thus trivial. We conclude that $N_{l/X} \isom \OO_{\PP^1}(-1) \oplus \OO_{\PP^1}(3-m)$.
\end{proof}

\begin{lemma}\label{intersection numbers}
	Let $C$ be a smooth rational curve lying in the smooth locus of a $3$-fold $X$ with normal bundle ${N_{C/X} \isom \OO_{C}(\alpha) \oplus \OO_C(\beta)}$ for some $\alpha \geq \beta \in \ZZ$.   Let $E\subset X' \maps{p} C \subset X$ be the blowup of $C$ with exceptional divisor $E = \PP(N_{C/X}) \isom \FF_{\alpha - \beta}$ and let  $C'$ be the unique negative section of $E \map C$ or a $0$-section if $E\isom \FF_0$.
	\begin{enumerate}
		\item We have $E\cdot C' = \alpha$;  in particular $E|_E \sim_E -C' + \beta f$.	
		\item Suppose that $S \subset X$ is a surface containing $C$ which is smooth along $C$ and that $(C^2)_S = \kappa$. If $S'$ is the strict transform of $S$, then  $D \defeq S'|_E \sim_E C' + (\alpha - \kappa)f$.
		\item If ${N_{C'/X'} = \OO_{C}(\alpha') \oplus \OO_C(\beta')}$ with $\alpha' \geq \beta'$, then $\alpha' + \beta' = \beta$. If furthermore $2\alpha - \beta < 2$ then  $\alpha' = \beta - \alpha$ and $\beta' = \alpha$.
	\end{enumerate}	
\end{lemma}

\begin{proof}
	We first use the adjunction formula on $C' \subset E$ and $C \subset X$ to obtain 
	\[
	K_E\cdot C' = 2g - 2 - (C')_E^2 \, \text{ and } \, K_X\cdot  C  = 2g - 2 - \deg N_{C/X} = 2g - 2 - (\alpha + \beta)
	\] 
	respectively. We have
	\begin{gather*}
	K_E = (K_{X'} + E)|_E = (p^*K_X + 2E)|_E \implies K_E \cdot C' = p^*K_X\cdot C' + 2E\cdot C' = K_X \cdot C + 2E\cdot C'\\
	\implies E\cdot C' = \frac{K_E \cdot C' - K_X \cdot C}{2} = \frac{2g-2 + \alpha - \beta - (2g-2) + \alpha + \beta}{2} = \alpha.
	\end{gather*}
	We can then write $E|_E \sim_E kC' + lf$ and use the facts that $E\cdot f = -1$ and $E \cdot C' = \alpha$ to deduce that $k = -1$ and $l = \beta$. This is (1).
	
	Since $S$ is smooth along $C$, $D\subset S' \maps{p} C \subset S$ is an isomorphism. We write $D \sim_E k C' + l f$. 	
	By $S$ being smooth (hence of multiplicity 1) along $C$ we get $k = 1$. We also have
	\begin{gather*}
	\kappa = (D)^2_{S'} = D \cdot E = (C' + lf)(-C' + \beta f) =  \alpha - \beta + \beta - l \implies  l = \alpha - \kappa.
	\end{gather*}
	This is (2).
	
	Finally, we have
	\begin{gather*}
	\deg N_{C'/X'} = 2g - 2 - K_{X'} \cdot C' = 2g - 2 - (p^*K_X + E)\cdot C' =\\
	(2g - 2 - K_X\cdot C) - E\cdot C' = \deg N_{C/X}  - \alpha \\
	\implies \alpha' + \beta' = \alpha + \beta  - \alpha = \beta.
	\end{gather*}
	
	Moreover, since $C' \subset E$ and $E \subset X'$ are regular imbeddings, the normal bundle sequence (see \cite[Proposition 3.4]{Fulton-Lang}) 	yields
	\[
	0 \map N_{C'/E} \map N_{C'/X'} \map N_{E/X'}|_{C'} \map 0
	\]
	which is actually
	\[
	0 \map \OO_{C'}(\beta - \alpha) \map \OO_{C'}(\alpha') \oplus \OO_{C'}(\beta') \map \OO_{C'}(\alpha) \map 0.
	\]
	Such extensions are classified by
	\[
	\Ext(\OO_{C'}(\beta - \alpha),\OO_{C'}(\alpha)) \isom H^1(\PP^1,\OO(\beta - 2\alpha)) \isom H^0(\PP^1, \OO(-2 + 2\alpha - \beta)).
	\]
	If $2\alpha - \beta < 2$, then $\Ext(\OO_{C'}(\beta - \alpha),\OO_{C'}(\alpha))  = 0$, thus the extension above is trivial and we deduce (3).
\end{proof}

\begin{lemma}\label{ampleness}
	Let $N$ be a nef divisor such that $N^{\perp} = \RR_{\geq 0} C$, where $C$ is the numerical class of a curve and let $E$ be any divisor with $E\cdot C <0$. Then there exists some $r_0 > 0$ such that for each $r \geq r_0$ the divisor $rN- E$ is ample.
\end{lemma}

\begin{proof}
	First we fix a norm $\norm{\cdot}$ on $N_1(X)$ such that $\norm{C} = 1$. 
	We consider the linear functionals
	\[
	\begin{array}{ccccccc}
	n\colon N_1(X) &\to  & \RR, & &e\colon N_1(X) &\to  & \RR\\
	C & \mapsto & N\cdot C & & C & \mapsto & E\cdot C.
	\end{array}
	\]
	Since $N_1(X)$ is a finite dimensional vector space, the functionals $n,e$ are continuous.
	
	Let $U$ be a neighbourhood of $C$ such that for all $x \in U$, $e(x) <0$.  Then for every $x \in U\cap \overline{\NE}(X)$ and $\epsilon \geq 0$ we have $\left(N - \epsilon E\right)x > 0$.
	
	On the other hand, the set $S = \left(S^1 \setminus U\right)\cap \overline{\NE}(X)$ is a closed subset of the compact set $S_1$, thus compact. This implies that $n(S)$ is also compact and is contained in $(0,+\infty)$, since $\RR_{\geq 0}C \cap S = \emptyset$. Let $m$ be the minimum of $n(S)$. Then for any $x \in S$ and $\epsilon > 0$ we have
	\[
	\left(N - \epsilon E\right)x = N\cdot y - \epsilon E\cdot y \geq m - \epsilon \norm{E}\norm{y} = m - \epsilon \norm{E}.
	\]
	Then for every $\epsilon \leq \epsilon_0 = \frac{m}{\norm{E}}$, $N - \epsilon E$ is strictly positive on $\overline{\NE}(X) \cap S^1$ and thus on the whole $\overline{\NE}(X)$. Finally we set $r_0 = \frac{1}{\epsilon_0}$, $r = \frac{1}{\epsilon}$ and multiply $N - \epsilon E$ by $r$. Kleiman's criterion for ampleness yields the desired result.
\end{proof}

Next we prove a base-point free type lemma.

\begin{lemma}\label{bpf}
	Let $C$ be a smooth rational curve lying on a smooth surface $E$ in the smooth locus of a $3$-fold $X$. We assume that: 
	\begin{itemize}
		\item $C$ generates an extremal ray of $\overline{\NE}(X)$
		\item $E \isom \FF_n$;
		\item $E \cdot C < 0$;
		\item $C\subset E$ is the unique $(-n)$-section of $\FF_n$;
		\item $C$ is the unique irreducible curve on $X$ whose numerical class lies on $\RR_{\geq 0}[C] \subset \overline{\NE}(X)$.
	\end{itemize}
	Then $C$ is contractible, i.e.\ there exists a birational morphism $X \map X'$ that contracts $C$ and only $C$.
\end{lemma}

\begin{proof}
	Since $C$ generates an extremal ray of $\overline{\NE}(X)$, there exists a nef divisor $N$ such that $N^{\perp} = \RR_{\geq 0}[C]$. By Lemma \ref{ampleness}, the divisor $A \defeq rN - E$ is ample for $r \gg 0$. Then there exists some $k_0 \in \NN$ such that $h^1(X,kA) = 0$ for every $k\geq k_0$. We fix such $k$  and using induction on $i$ we will prove that $h^1(X,kA+iE)= 0$ for $0\leq i \leq k-1$.
	
	For the base case $i = 0$ we get $h^1(X,kA) = 0$ by our choice of $A$ and $k$. For the inductive step we assume that $h^1\left(X,kA + (i-1)E\right) = 0$ and consider the exact sequence
	\[
	H^1(X,\OO(kA + (i-1)E)) \map H^1(X,\OO(kA + iE)) \map H^1(E,\OO(kA + iE)).
	\]
	By possibly repicking $r$ even larger, again using Lemma \ref{ampleness}, we may assume that $kA+iE$ is positive against any fibre of $E \isom \FF_n \map \PP^1$. Moreover it is also positive against $C$ and thus using Lemma \ref{cohomology Fn} we  get $h^1(E,kA + iE) = 0$. Consequently, using the inductive hypothesis we get that $h^1(X,kA + iE) = 0$ for $0\leq i \leq k-1$. Especially, for $i = k-1$ we get $h^1(X,krN-E) = 0$.
	
	Notice that $rN = A + E$ where $A$ is ample and $E$ is effective. This implies that among the global sections of $rN$ there are sections of the form $s_i = a_ie_i \in H^0(X,A) \otimes H^0(X,E) \subseteq H^0(X,rN)$. Since $A$ is ample, this immediately implies that $rN$ has no base points away from $E$.
	
	As for any base points on $E$ we consider the exact sequence
	\[
	H^0(X,\OO(krN)) \map H^0(E,\OO(krN)) \map H^1(X,\OO(krN - E)) = 0.
	\]
	Then $krN|_E$ is a nef divisor on a smooth surface, zero only against a rational curve of negative self-intersection. Consequently, it is semi-ample (see for example the proof of \cite[Theorem 1-1-6]{Matsuki}). Finally using the exact sequence above we may lift sections of $krN|_E$ to sections of $krN$, proving that the stable base locus of $N$ does not meet $E$. Thus $N$ is semiample.
\end{proof}

\begin{proposition}[The $(-1,-m)$-FLIP]\label{1,m flip}
	Let $C$ be smooth rational curve lying in the smooth locus of a $3$-fold $X$. Suppose that $X \map Z$ is a contraction morphism, contracting only $C$ and that the normal bundle $N_{C/X}$ of $C$ in $X$ is isomorphic to $\OO_{\PP^1}(-1) \oplus \OO_{\PP^1}(-m)$ with $1 \leq m \leq 3$. 
	Then the anti-flip of $C$ exists (i.e.\ there exists an SQM of $X$ over $Z$ centred at $C$ and the target variety has at worst terminal singularities).
\end{proposition}

\begin{proof}
	The cases $m = 1$ and $2$ are the classical cases of the Atiyah flop and the Francia flip respectively. We refer to \cite[6.10, p.\ 162]{Debarre} for the explicit construction of the resolution of the anti-flip.
	
	Suppose that $m = 3$. Write $Y \map X$ for the blowup of $C$ with exceptional divisor $E \isom \FF_{m-1} = \FF_{2}$. Denote by $\sigma$ the $(-2)$-section and by $f$ a fibre of $E \map C$. Then the relative cone of curves $\NE(Y/Z)$ is generated by the classes of $\sigma$ and $f$ and by Lemma \ref{bpf}, $\sigma$ is contractible.	
	
	By Lemma \ref{intersection numbers} the normal bundle of $\sigma$ in $Y$ is $\OO_{\PP^1}(-1)\oplus \OO_{\PP^1}(-2)$ and so the anti-flip $\chi \colon Y \psmap Y'$ of $\sigma$ exists. More specifically, $\chi$ is the inverse of a Francia flip. If $Y \map \hat Z$ is the contraction of $\sigma$, we have the diagram
	\[
	\xymatrix@R=0.2cm@C=0.4cm{
		Y \ar@{..>}[rr]^{\chi} \ar[rd] \ar[dd] && {Y}' \ar[ld]\\
		& \hat Z \ar[dd]\\
		X \ar[rd]\\
		& Z
	}
	\]
	where $\hat Z \map Z$ is the morphism induced by the inclusion $\overline{\NE}(Y/\hat{Z}) \subset \overline{\NE}(Y/Z)$.
	Using the explicit resolution of the Francia flip in \cite[6.10, p.\ 162]{Debarre} in conjunction with Lemma \ref{intersection numbers}(2), one can check that the restriction of $\chi$ on $E$ is the contraction of the $(-2)$-curve.	
	The relative cone of curves $\NE(Y'/Z)$ is generated, over $\QQ$, by the classes of the anti-flipped curve as well as the class of any curve in the strict transform $E'$ of $E$. This is because $E' \isom \bar{\FF}_2$ and so, by Lemma \ref{FF_n bar}, $\rho(E') = 1$, thus the numerical class of any curve on it covers $E'$. Thus if $\sigma_+$ is a section of $E$ disconnected from $\sigma$, $\sigma_+' \defeq \chi(\sigma_+)$ generates an extremal ray of $\NE(Y')$. Furthermore since we chose $\sigma_+'$ to be disconnected from the centre of $\chi$ we have
	\[
	K_{Y'}\cdot \sigma_+' = K_Y\cdot \sigma_+  = K_Y\cdot (\sigma + 2f) = 1 - 2 = -1.
	\]
	Thus $\sigma_+'$ is contractible by a divisorial contraction $Y' \map X'$ and since it's $K_{Y'}$-negative and $Y'$ has terminal singularities, then so does $X'$. The diagram above becomes
	\[
	\xymatrix@R=0.2cm@C=0.4cm{
		Y \ar@{..>}[rr]^{\chi} \ar[rd] \ar[dd] && {Y}' \ar[ld] \ar[dd]\\
		& \hat Z \\
		X \ar[rd] \ar@{..>}[rr]^{\psi}&& X'\ar[ld]\\
		& Z
	}
	\]	
	where $\psi$ is the birational map induced by the diagram and actually the required pseudo-isomorphism. 
\end{proof}

Schematically the resolution described in the proof above looks as follows
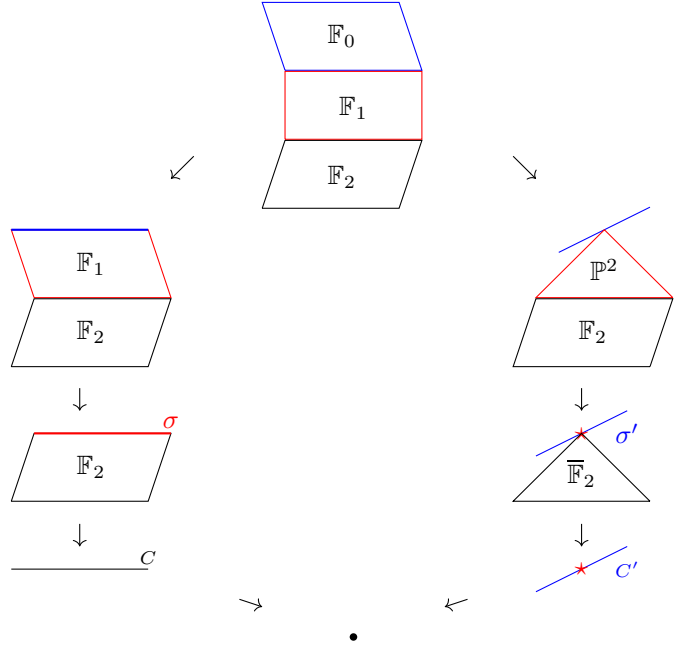
\begin{center}
	\begin{tikzpicture}[scale=0.3]
	\node at (0,0) [circle,fill,inner sep=1pt]{};

	\draw[->]  (-5,5/3) -- (-4,4/3);

	\draw (-15,3) -- (-9,3);
	\node at (-9,3.5) {$\scriptstyle{C}$};

	\draw[->] (-12,5) -- (-12,4);

	\draw (-15,6) -- (-14,9) -- (-8,9) -- (-9,6) -- (-15,6);
	\draw[red,thick] (-14,9) -- (-8,9);
	\node[red] at (-8,9.5) {$\sigma$};
	\node at (-11.5,7.5) () {$\FF_2$};

	\draw[->] (-12,11) -- (-12,10);

	\draw[red] (-14,15) -- (-15,18) -- (-9,18) -- (-8,15) -- (-14,15);
	\draw[blue,thick] (-15,18) -- (-9,18);	
	\node at (-11.5,16.5) () {$\FF_1$};	
	\draw (-15,11.95) -- (-14,14.95) -- (-8,14.95) -- (-9,11.95) -- (-15,11.95);
	\node at (-11.5,13.5) () {$\FF_2$};

	\draw[->] (-7,21.25) -- (-8,20.25);

	\draw[blue] (-3,25.05) -- (-4,28.05) -- (2,28.05) -- (3,25.05) -- (-3,25.05);
	\node at (-0.5,26.55) {$\FF_0$};
	\draw[red] (-3,22) -- (-3,25) -- (3,25) -- (3,22)-- (-3,22);
	\node at (-0,23.4) {$\FF_1$};
	\draw (-4,18.95) -- (-3,21.95) -- (3,21.95) -- (2,18.95) -- (-4,18.95);
	\node at (-0.5,20.35) {$\FF_2$};
	
	\draw[->] (7,21.25) -- (8,20.25);

	\draw[blue] (9,17)-- (11,18) -- (13,19);
	\draw[red] (14,15) -- (11,18) -- (8,15) -- (14,15);
	\draw (13,11.95) -- (14,14.95) -- (8,14.95) -- (7,11.95) -- (13,11.95);
	\node at (10.5,13.5) () {$\FF_2$};
	\node at (11,16.25) () {$\PP^2$};	
	
	\draw[->] (10,11) -- (10,10);

	\draw[blue] (8,8) -- (10,9) -- (12,10);
	\node[blue] at (12,9) {$\sigma'$};
	\draw (13,6) -- (10,9)-- (7,6) -- (13,6);
	\node[on above layer, red] at (10,9) {$\bullet$};
	\node at (10,7.25) () {$\overline{\FF}_2$};

	\draw[->] (10,5) -- (10,4);

	\draw[blue] (8,2) -- (10,3) -- (12,4);
	\node[on above layer, red] at (10,3) {$\bullet$};
	\node[blue] at (12,3) {$\scriptstyle{C'}$};	
	
	\draw[->]  (5,5/3) -- (4,4/3);
	
	\end{tikzpicture}
\end{center}

\noindent where the dot represents the terminal point of $X'$ which is actually a quotient singularity of type $\frac{1}{3}(1,1,2)$.

\subsection{Conclusion}

\begin{proposition}\label{4-secant lines}
	Let $C\subset S$ be one of the curves in Table \ref{List of candidate curves} so that the strict transforms of the maximal secant lines have normal bundle $\OO_{\PP^1}(-1)\oplus \OO_{\PP^1}(-m)$, with $1\leq m \leq 3$. Let $\chi\colon X \psmap X'$ be the anti-flip of those strict transforms as constructed in Proposition \ref{1,m flip}. 
	
	If $C$ admits no $4$-secant lines, then $X'$ is Fano. Otherwise, $X'$ is weak-Fano and the strict transforms of $4$-secant lines generate (and are the only irreducible curves whose class is contained in) an extremal ray of $\NE(X')$.
\end{proposition}

\begin{proof}
	Let $f'\colon X' \rmap Z$ be the anti-canonical model of $X'$. Since $\chi$ is a pseudo isomorphism, we have
	\[
	Z = \Proj\left(\bigoplus_{m\in \NN}H^0(X',-mK_{X'})\right) = \Proj\left(\bigoplus_{m\in \NN}H^0(X,-mK_{X})\right).
	\]
	Now any curve $c'$ in $X'$ such that $-K_{X'}\cdot c' \leq 0$ is either contracted by or in the base locus of $f'$. 
	If $f \defeq f'\circ \chi$ and $c$ is the strict transform
	%\footnote{Note that any such $c'$ is not $\chi^{-1}$-exceptional. That is because $\chi^{-1}$ is a $K_{X'}$-flip and so the $\chi^{-1}$-exceptional curves are $K_{X'}$-negative.}
	of $c'$ under $\chi^{-1}$ then, by the equality above, $c$ is again either contracted by or in the base locus of $f$.  Thus $-K_X\cdot c \leq 0$. Assume for a second that those curves are exactly the $4$, $5$ or $6$-secant lines. Then a look at the table shows that when $C$ admits a $6$-secant line then it admits no $5$-secant lines. Thus, after the anti-flip, $X'$ is Fano if and only if $C$ admits no $4$-secant lines. Moreover, if $C$ admits $4$-secant lines, then their strict transforms are the only $K_{X'}$-zero curves.
	
	To complete the proof we have to show that the only curves with $-K_X\cdot c \leq 0$ are  the $4$, $5$ or $6$-secant lines. Let $c$ be such a curve, i.e.\ $c \sim dl-mf$ with $\frac{m}{d} \geq 4$, which is not a line.	
	Then $c$ is contained in the cubic surface $S$ and so we may write 
	\[
	c \sim_S l_1 + \dots + l_d,
	\]
	with $l_i \sim l-m_i f$ being lines in $S$.
	We first note that if for some $1\leq i \leq d$, $l_i$ does not meet any of the other lines in the decomposition then $(c\cdot l_i)_S \leq -1$ and thus $c$ is not irreducible. Thus we may assume that every line intersects another one. If all $l_i$ were $4$-secants then, by Lemma \ref{term}, they wouldn't intersect each other and so, for $\frac{m}{d}$ to be greater than or equal to $4$, we may assume that some of the lines in the decomposition of $c$ are $5$ or $6$-secants.
	From Table \ref{List of candidate curves}, we see that in all cases there are at most two $5$ or $6$-secant lines. We rearrange the decomposition of $c$ in the following way:
	\[
	c \sim_S a_1 l_1 + a_2l_2 + \sum_i^{e}l^{12}_i + \sum_j^{r_1}l^1_j + \sum_k^{r_2}l^2_k  + \sum_n^D l^0_n,
	\]
	where $D= d - a_1 - a_2 -e - r_1 - r_2$ and $l_1$ and $l_2$ are the two $5$ or $6$-secants; 
	the lines $l^{12}_i$ are the lines that meet both $l_1$ and $l_2$; 
	the lines $l^1_j$ meet $l_1$ but not $l_2$; 
	the lines $l^2_k$ meet $l_2$ but not $l_1$; 
	the lines $l^0_n$ meet none of the $l_1$ and $l_2$. 
	If there is only one $5$ or $6$-secant, we simply choose $a_2 = 0$ and get empty sums for $l_i^{12}$ and $l_k^2$.
	Intersecting with $l_1$ and $l_2$ respectively we get $a_1 \leq e + r_1$ and $a_2 \leq e + r_2$. By Lemma \ref{term}, for any $j$ and $k$, we have $C\cdot(l_1 + l^1_j), C\cdot\left(l_2 + l^2_k\right) \leq 7$. Similarly, checking Table \ref{List of candidate curves} we see that for any $i$,  we have $C\cdot(l_1 + l_2 + l^{12}_i)\leq 11$. Finally, we have
	\begin{align*}
	m = C\cdot c &= C\cdot \left(a_1l_1 + a_2l_2 + \sum_i^{e}l^{12}_i + \sum_j^{r_1}l^1_j + \sum_k^{r_2}l^2_k  + \sum_n^D l^0_n\right)\\
	&\leq C \cdot \left((e+r_1)l_1 + (e+r_2)l_2 + \sum_i^{e}l^{12}_i + \sum_j^{r_1}l^1_j + \sum_k^{r_2}l^2_k  + \sum_n^D l^0_n\right)\\
	&=C\cdot \left(\sum_i^{e} (l_1 + l_2 + l^{12}_i) + \sum_j^{r_1} (l_1 + l^1_j) + \sum_k^{r_2} (l_2 + l^2_k) +  \sum_n^D l^0_n \right)\\
	&\leq 11e + 7(r_1+r_2) + 4D\leq 4d - e - r_1 -r_2 < 4d,
	\end{align*}
	which contradicts $\frac{m}{d} \geq 4$.
\end{proof}

\begin{remark}\label{sum of lines}
	In cases $5$ and $6$ of  Table \ref{List of candidate curves}, a similar argument as the one presented in the proof above, gives us a more precise result. Namely, that any curve with $-K_X\cdot c < 1$ is a $4$, $5$ or $6$-secant line. 
	
	Indeed, we may repeat the calculation above but instead of singling out only the $5$-secant lines, we single out the $4$-secants as well. We then only need to observe that in those cases, any two intersecting lines we $l_1$ and $l_2$ we have $C\cdot (l_1 + l_2 ) \leq 6$ and for any line $l_3$ joining a $4$-secant $l_1$ and a $5$-secant $l_2$ we have $C\cdot (l_1 + l_2 + l_3) \leq 9$.
\end{remark}

\begin{theorem}
	Let $C$ be a space curve lying on a smooth cubic surface $S$ such that $X \defeq \Bl_{C}\PP^3$ is not weak Fano. Then $C$ induces a Sarkisov link if and only if its class on $S$ appears in Table \ref{List of candidate curves} (up to the assumptions of Remark $\ref{assumptions on the type}$).
\end{theorem}

\begin{proof}
	The first implication is clear since the curves appearing in Table \ref{List of candidate curves} are exactly those satisfying the necessary conditions of Proposition \ref{necessary conditions}.
	
	Conversely, assume that $C$ is one of the curves in the Table. We first note that in all cases Proposition \ref{X is an MDS} implies that the blowup of $C$ is always an MDS. In turn, by Proposition \ref{Induces link iff MDS}, we are guaranteed the existence of a Sarkisov link as long as the varieties produced by the $2$-ray game are terminal, which we now check: 
	By Propositions \ref{normal bundles of lines} and \ref{1,m flip} the anti-flip of the $5$ or $6$-secant lines produces a terminal 3-fold. Furthermore, by Proposition \ref{4-secant lines} the 3-fold is (weak-)Fano. Thus any further step in the $2$-ray game is $K$-non-positive and so retains the terminal singularities.
\end{proof}

\section{Study of the links}\label{Construction and study of the links}

In this final section we aim to finish the construction of the links produced in the previous section. We also calculate some invariants of the targets of the links such as their singularities and the cube of the anti-canonical divisor.

\subsection{Some preliminary calculations}

\begin{lemma}[{\cite[Lemma 2.4]{WeakFanos}}]\label{cube anticanonical}
	Let $C\subset Y$ be a smooth curve of genus $g$ in a smooth $3$-fold and let $\pi\colon X \map Y$ be the blowup of $C$. Then
	\[
	(-K_X)^3 = (-K_Y)^3 + 2(-K_Y)C - 2 + 2g. 
	\]
	In particular, if $Y = \PP^3$ and $C$ is a curve of degree $d$, we have $(-K_X)^3 = 62 - 8d + 2g$.
\end{lemma}

In the following Lemma we use the notation introduced in Lemma \ref{FF_n bar}.

\begin{lemma}\label{cube of the anticanonical, contraction}
	Let $X \maps{p} Y$ be a divisorial contraction to a point, between $\QQ$-factorial terminal threefolds with exceptional divisor $E$.	
	If $E$ lies on the smooth locus of $X$ and $K_X = p^*K_Y + \alpha E$ is the ramification formula then
	
	if $E \isom \PP^1 \times \PP^1$ with normal bundle $N_{E/X} = \OO_{\PP^1 \times \PP^1}(-1,-1)$, $\alpha = 1$;
	
	if $E \isom \PP^2$ with normal bundle $N_{E/X} = \OO_{\PP^2}(-2)$, $\alpha = \frac{1}{2}$.\\	 
	Moreover, we have $-K_Y^3 = -K_X^3 + 2$, $-K_Y^3 = -K_X^3 + \frac{1}{2}$ in the two cases respectively.

	Similarly, without any extra assumptions on the singularities this time, if $E \isom \overline{\FF}_2$ with normal sheaf $N_{E/X} = \OO_{\overline{\FF}_2}(-3\bar{f})$ and $K_X\cdot \bar{\sigma} = -1$, then $\alpha = \frac{1}{3}$. Moreover, we have $-K_Y^3 = -K_X^3 + \frac{1}{6}$.
\end{lemma}

\begin{proof}
	We first compute the discrepancy. Denote by $l$ a line in $E$, if $E \isom \PP^2$ or any ruling of $E$, if $E\isom \PP^1 \times \PP^1$. By the adjunction formula for $E$ we have $K_E = (K_Y + E)|_E$. 	Intersecting both formulas with $l$ and solving for $\alpha$ we get
	\[
	\alpha = \frac{K_E\cdot l  - E\cdot l}{E\cdot l}.
	\]
	In the first case we have $K_E\cdot l = -2$ and $E\cdot l = -1$ giving us $\alpha = 1$. In the second case we have $K_E\cdot l = -3$ and $E\cdot l = -2$ and we get $\alpha = \frac{1}{2}$.
	
	In the third case, intersecting the ramification formula for $p$ with $\bar{\sigma}$ we get
	\[
	\alpha = \frac{K_X\cdot \bar{\sigma} - p^*K_Y \cdot \bar{\sigma}}{E\cdot \bar{\sigma}} = \frac{ -1 - p^*K_Y \cdot \bar{\sigma}}{(-3\bar{f})\bar{\sigma}},
	\]
	which, since $\bar{\sigma}$ is $p$-exceptional, equals to $\frac{1}{3}$.	
	
	Finally we have
	\[
	K_X^3 = (p^*K_Y + \alpha E)^3 = K_Y^3 + 3\alpha(p^*K_Y)^2E + 4\alpha^2(p^*K_Y)E^2 + \alpha^3 E^3.
	\]
	In all cases the middle terms vanish.
	In the first case we have $\alpha = 1$ and $E^3 = 2$, in the second case we have $\alpha = \frac{1}{2}$ and $E^3 = 4$ and in the last case we have $\alpha = \frac{1}{3}$ and $E^3 = \frac{9}{2}$ and so a computation completes the proof.
\end{proof}

\begin{corollary}\label{cube of the anticanonical, flip}
	Let $X \psmap X'$ be a $(1,2)$-flip. Then $-K_{X'}^3 = -K_X^3 + k\frac{1}{2}$ where $k$ is the number of flipped curves.
	
	Similarly, if $X \psmap X'$ is a $(1,3)$-flip, then $-K_{X'}^3 = -K_X^3 + k\frac{8}{3}$ where $k$ is the number of flipped curves.
\end{corollary}

\begin{proof}
	For simplicity we will assume that the number of flipped curves is $1$. The general case is similar.
	
	We first treat the $(1,2)$-flip. Consider the resolution
	\[
	\xymatrix@C=0.2cm@R=0.5cm{
		&& Y_0 \ar[lld]_{r_0} \ar[rrd]^{s_0} \\
		Y_1 \ar[d]_{r_1} &&&& Y_1' \ar[d]^{s_1}\\
		X  \ar@{..>}[rrrr] &&&& X'
	}
	\]
	as constructed in \cite[6.10, p.\ 162]{Debarre}. By \cite[Lemma 2.4]{WeakFanos} we have 
	\[
	-K_{Y_1}^3 = -K_X^3 + 2K_X\cdot C - 2 +2g
	\]
	where $C$ is the flipped curve and $g$ is its genus. Since $K_X\cdot C = 1$ and $g=0$ we get $-K_{Y_1}^3 = -K_X^3$.
	We also have $-{K_{Y_1}}^3 = -{r_0^*K_{Y_1}}^3 = (-K_{Y_0} +E)^3 = -{s_0^*K_{Y_1'}}^3 = -K_{Y_1'}^3$. Finally, by Lemma \ref{cube of the anticanonical, contraction} we have $-K_{X'}^3 = -K_{Y_1'}^3 + \frac{1}{2} = -K_X^3 + \frac{1}{2}$. 
	
	For the $(1,3)$-flip we again consider the resolution 
	\[
	\xymatrix@C=1.2cm@R=0.8cm{
		Y \ar@{..>}[r]^{\chi} \ar[d] & Y'\ar[d]\\
		X \ar@{..>}[r]				& X'
	}
	\]
	constructed in Proposition \ref{1,m flip}, where now $\chi$ is a $(1,2)$-flip.  Using the formula $-K_{Y}^3 = -K_X^3 + 2K_X\cdot C - 2 +2g$ and since $K_X\cdot C =2$ we get $-K_Y^3 = -K_X^3 +2$.
	By the previous statement and by Lemma \ref{cube of the anticanonical, contraction}, following the diagram counter clockwise we get $-K_{X'}^3 = -K_{Y'}^3 +\frac{1}{6} = -K_Y^3 + \frac{1}{2} + \frac{1}{6} = -K_X^3 + 2 + \frac{1}{2} + \frac{1}{6} = -K_X^3 + \frac{8}{3}.$
\end{proof}

\subsection{Some properties of the links}\leavevmode

\subsubsection{Dimension of linear system of cubics}
A careful examination of Table \ref{List of candidate curves} reveals that the curves $1$ through $4$, all admit a pencil of $7$-secant conics. Indeed, using Lemma \ref{conic bundle} this amounts to finding two distinct, intersecting, $m_1$ and $m_2$-secant lines such that $m_1 + m_2 = 7$. For example, a pair of such lines that works in all cases is $c_1$ and $l_{1,5}$. This immediately implies that in those cases, $C$ is contained in a unique cubic.

However, this is not true for the last two cases of the table as the following Lemma shows.

\begin{lemma}\label{pencil of cubics}
	Let $C$ be one of the last two curves of Table \ref{List of candidate curves}. Then $C$ is contained in a pencil of cubics.
\end{lemma}

\begin{proof}
	We consider the exact sequence
	\[
	0 \map H^0(X,\OO_X) \map H^0(X,\OO_X(3H-E)) \map H^0(X,\OO_S(3H-E|_S)) \map H^1(X,\OO_X).
	\]
	Since $X$ is projective, rational with rational singularities, we have $h^0(X,\OO_X) = 1$ and $h^1(X,\OO_X) = h^1(\PP^3,\OO_{\PP^3}) = 0$. Thus 
	\[
	h^0(X,\OO_X(3H-E)) = h^0(X,\OO_S(3H-E|_S)) +1,
	\]
	where the sections on the left-hand side of this equation correspond to cubics containing $C$. Then one can check that the divisor $3H-E|_S$ is effective and fixed. We do this calculation for the case $5$ of the table.
	
	We first note that $3H-E|_S$ has negative intersection with and thus contains all $4$ and $5$-secant lines in its base locus with multiplicity $1$ and $2$ respectively. We then have
	\begin{gather*}
	3H-E|_S - 2c_1 - c_2 = -3K_S - C - 2c_1 - c_2 = 0
	\end{gather*}
	and so the movable part of $3H-E|_S$ is zero, i.e.\ $h^0(X,\OO_S(3H-E|_S))=  h^0(X,\OO_S)= 1$, proving the claim.
\end{proof}

\subsubsection{Mori chambers}
Let 
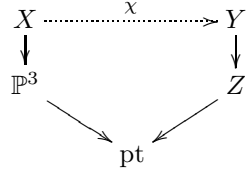
\[
\xymatrix@R=0.4cm{
	X \ar[d] \ar@{..>}[rr]^{\chi} && Y \ar[d]\\
	\PP^3 \ar[rd] 				&    			& Z\ar[ld]\\
	& \text{pt}
}
\]
be a Sarkisov diagram, where $X \map \PP^3$ is the blowup of one of the curves of Table \ref{List of candidate curves} and $\chi$ is a pseudo-isomorphism which is a composition of anti-flips, flops and flips ($Y \map Z$ can be either divisorial or of fibre type). 

We have already proven that $X$ is an MDS and so by Proposition \ref{Properties of MDS} the pseudo-effective cone of $X$ admits a decomposition into chambers of the form 
\[
\mathcal{C}_i = g_i^*\Nef(Y_i) + \RR_{\geq 0}\{E_1, \dots, E_k\},
\]
where $g_i$ are birational contractions and $E_j$ are prime divisors contracted by $g_i$. Moreover, if $\mathcal{C}_i$ and $\mathcal{C}_j$ are neighbouring chambers, $Y_i$ and $Y_j$ are connected by an SQM or an extremal contraction.

Since $\rho(X)=2$, the pseudo-effective cone $\overline{\Eff}(X)$ is 2-dimensional. In the following Lemma we will compute $\overline{\Eff}(X)$.

\begin{lemma}\label{Pseff}
	If $C$ is one of the curves of Table \ref{List of candidate curves}, then $\overline{\Eff}(X)$ is spanned by the divisors $E$ and $3H-E$.
\end{lemma}

\begin{proof}
	By Lemma \ref{curves}, $\overline{\Eff}(X)$ is spanned by $E$ and a divisor $dH-mE$ with $\frac{m}{d}$ maximal. We first note that in all $6$ cases 
	the pencil of conics $\mathcal{P}$ associated to the lines $c_1$ and $l_{1,5}$ (see Lemma \ref{conic bundle}) is a pencil of $6$ or $7$-secant conics (depending on the case), which spans a cubic containing $C$. 
	If now $D\sim dH-mE$ is an effective divisor with $\frac{m}{d} > \frac{1}{3}$ whose support is irreducible, then using B\'ezout's theorem, we deduce that all conics in $\mathcal{P}$ are contained in $D$. Since $\mathcal{P}$ spans a cubic $S$, we have $S \subseteq D$. Since $D$ is irreducible we get $S = D$.
\end{proof}

We then have the following dichotomy:
\begin{itemize}
	\item As discussed in the beginning of this subsection, curves in cases $1$ through $4$ of Table \ref{List of candidate curves} are contained in a unique cubic surface $S$ , thus the morphism associated to the linear system of cubics is $X \map pt$. This implies that the rightmost chamber $\mathcal{C}$, which is the one associated to $Z$, contains only big divisors (since they all are a positive combination of an ample and an effective divisor). Thus $Z$ is birational to $X$ and so it is a Fano $3$-fold of Picard rank $1$. Moreover $Y\map Z$ is the contraction of the strict transform of  $S$. This is a link of \textbf{Type II}.
	\item In cases $5$ and $6$, Lemma \ref{pencil of cubics} yields that the rational map associated to the system of cubics is $X \map \PP^1$. This implies that $Z = \PP^1$ and in turn that $Y$ is a del-Pezzo fibration with fibres the strict transforms of the cubics containing $C$. This is a link of \textbf{Type I}.
\end{itemize}

\subsubsection{The matrix of the transformation}\label{section of the matrix}
The pseudo-isomorphism $\chi$ induces an isomorphism $\chi_*$ between the groups $\WDiv(X)$ and $\WDiv(Y)$. The isomorphism is given by restricting prime divisors in the regular locus of $\chi$ and taking the closure of their image in $Y_0$. We may extend this to an isomorphism between the $\QQ$-vector spaces $\WDiv_{\QQ}(X)$ and $\WDiv_{\QQ}(Y)$.

We fix a cubic $S$ containing $C$ and define $T\defeq \chi_*S$. We also fix the bases $\WDiv_{\QQ}(X) = \langle H,E \rangle$ and $\WDiv_{\QQ}(Y_0) = \langle K_{Y_0},T \rangle$. Note that these divisors do not necessarily generate the $\ZZ$-modules $\WDiv(X)$ and $\WDiv(Y_0)$.
Since $K_X = -4H + E \mapsto K_{Y_1}$ and $S = 3H - E \mapsto T$, the matrix of the isomorphism is 
$\begin{pmatrix}
-1 & -3 \\
-1 & -4
\end{pmatrix}$
with inverse 
$\begin{pmatrix*}[r]
-4 & 3 \\
1 & -1
\end{pmatrix*}$.

\subsection{Some invariants of the targets of the links -  Cases 1 through 4} 

As proven earlier, in the cases $1$-$4$ the Sarkisov diagram takes the form 
\[
\xymatrix@R=0.4cm{
	X \ar[d] \ar@{..>}[r] & X_0 \ar@{..>}[r] & Y_0 \ar[d]^{p}\\
	\PP^3 \ar[rd] 				&    			& Y\ar[ld]\\
	& \text{pt}
}
\]
where $Y$ is a Fano $3$-fold of Picard rank $1$.

\subsubsection{The contraction {$Y_0 \to Y$}} The restriction of the pseudo-isomorphism $\chi\colon X \psmap Y_0$ to the cubic $S$ is the contraction of the anti-fliped/flopped curves. This can be verified using the explicit resolutions of the $(1,m)$-flips of Proposition \ref{1,m flip} as well as Lemma \ref{intersection numbers}. 

In cases $1,2$ and $4$ the restriction $\chi|_S\colon S \map T$ is just the contraction of $6$ $(-1)$-curves and so $T \isom \PP^2$. More specifically, $\chi|_S$ fits into the diagram
\[
\xymatrix{
	&S \ar[ld]_p \ar[rd]^{\chi|_S}\\
	\PP^2 && T
}
\]
where $p$ is the blow up of $6$-points followed by the contraction of the $6$ conics through $5$ of the $6$ points. We may compute that the pullback of line $l \subset T \isom \PP^2$ on $S$ is of type $(5;2,2,2,2,2,2)$. We then have
\[
T|_T\cdot l = S|_S\cdot (\chi|_{S})^*l,
\]
which in all $3$ cases can be computed to be $-2$. Thus $T\isom \PP^2$ with $N_{T/Y_0} = \OO_{\PP^2}(-2)$.

In case $3$, the restriction $\chi|_S\colon S \map T$ contracts $5$ $(-1)$-curves, thus $T$ can be isomorphic to either $\FF_1$ or $\PP^1 \times \PP^1$. However, in reference to the morphism $S \map \PP^2$, $\chi|_S$ contracts the strict transforms of: 4 conics through 5 of the points; 1 line through 2 of the points.
This implies that $T$ is isomorphic to $\PP^1 \times \PP^1$. More specifically, $\chi|_S\colon S \map T$ factors as
\[
\xymatrix{
	& S \ar[ld] \ar[rd] \ar[rr] & & S' \ar[ld] \ar[rd] \\
	\PP^2 && T' \isom \PP^2 && T
}
\]
where starting from left to right the morphism are: the blowup of 6 points, the contraction of the 6 conics though 5 of the 6 points, the blowup of 2 points and finally the contraction of the line through the 2 points. 

Pulling back classes of the two rullings of $T$ under the morphism $S \map T$,  we find that they are the strict transforms of lines in $T'$ passing through the 2 blown up points. Their classes on $S$ are $(5;2,2,2,2,2,2) - c_5$ and $(5;2,2,2,2,2,2) - c_6$. As above, intersecting $S|_S$ with both those classes we get $-1$ and so we find that the normal bundle of $T$ in $Y_0$ is $\OO_{\PP^1 \times \PP^1}(-1,-1)$.

\subsubsection{Singularities of $Y$} Following the Sarkisov diagram clockwise we may compute the singularities of $Y$. We will do this case by case.
\begin{itemize}
	\item[$\#1.$] We have the $(1,2)$-flip of 2 curves $X \psmap X_0$, followed by the $(1,1)$-flop of 4 more curves $X_0 \psmap Y_0$ and finally the contraction of $T \isom \PP^2$ with $N_{T/Y_0} = \OO_{\PP^2}(-2)$. These modifications produce $2$ quotient singularities of type $\frac{1}{2}(1,1,1)$, no singularities and another quotient singularity of type $\frac{1}{2}(1,1,1)$ respectively.
	\item[$\#2.$] We have the $(1,3)$-flip of a curve, followed by the flop of $5$ curves and finally the contraction of $T \isom \PP^2$ with $N_{T/Y_0} = \OO_{\PP^2}(-2)$. These modifications produce $1$ quotient singularity of type $\frac{1}{3}(1,1,2)$, no singularities and a quotient singularity of type $\frac{1}{2}(1,1,1)$ respectively.
	\item[$\#3.$] We have the $(1,2)$-flip of a curve, followed by the flop of $4$ curves (notice that the line $l_{5,6}$ is a $4$-secant) and finally the contraction of $T \isom \PP^1 \times \PP^1$ with $N_{T/Y_0} = \OO_{\PP^1 \times \PP^1}(-1,-1)$. These modifications produce $1$ quotient singularity of type $\frac{1}{2}(1,1,1)$, no singularities and an ordinary double point respectively.
	\item[$\#4.$] We have the $(1,2)$-anti-flip of a curve, followed by flop of 5 curves and finally the contraction of $T \isom \PP^2$ with $N_{T/Y_0} = \OO_{\PP^2}(-2)$. These modifications produce $1$ quotient singularity of type $\frac{1}{2}(1,1,1)$, no singularities and a quotient singularity of type $\frac{1}{2}(1,1,1)$ respectively.
\end{itemize}

\subsubsection{Cube of $-K_Y$} Again following the Sarkisov diagram clockwise  and using Lemmata \ref{cube anticanonical} and \ref{cube of the anticanonical, contraction} as well as Corollary \ref{cube of the anticanonical, flip} we may compute that
\begin{itemize}
	\item[$\#1.$] $-K_X^3 = 8\phantom{0} \implies -K_{X_0}^3 = 9\,\,\, \implies -K_{Y_0}^3 = 9 \,\,\,\implies -K_Y^3 = \frac{19}{2}$;
	\item[$\#2.$] $-K_X^3 = 6\phantom{0} \implies -K_{X_0}^3 = \frac{26}{3} \implies -K_{Y_0}^3 = \frac{26}{3} \implies -K_Y^3 = \frac{55}{6}$;
	\item[$\#3.$] $-K_X^3 = 10 \implies -K_{X_0}^3 = \frac{21}{2} \implies -K_{Y_0}^3 = \frac{21}{2} \implies -K_Y^3 = \frac{25}{2}$;
	\item[$\#4.$] $-K_X^3 = 8\phantom{0} \implies -K_{X_0}^3 = \frac{17}{2} \implies -K_{Y_0}^3 = \frac{17}{2} \implies -K_Y^3 = 9$.
\end{itemize}

\subsubsection{Fano index of $Y$} 
We distinguish cases.\\
\underline{Cases $1$ and $4$}:
In those cases, the least common multiple among the indices of the singularities of $Y$ is $2$. This implies that $2\WDiv(Y) \subseteq \CDiv(Y)$ (see \cite[Corollary 5.2]{Index}). Denote by $r$ the Fano-Weil index of $Y$. That is precisely the number
\[
r \defeq \max\left\{q \in \ZZ \mid -K_Y \sim qA, A\text{ is a Weil divisor}\right\}.
\]
Let $A$  be a Weil divisor such that $-K_Y = rA$.
Then, using Lemma \ref{cube of the anticanonical, contraction}, we have
\[
rp^*(2A) = p^*(2rA) = p^*(-2K_Y) = -2K_{Y_0} + T,
\]
which is the vector $(-2,1)$ in the $\ZZ$-submodule $\langle K_{Y_0},T\rangle \leq \WDiv(Y_0)$. By the calculations in Subsection \ref{section of the matrix} its strict transform is
\[
r\chi_*^{-1}(p^*(2A)) = 
\begin{pmatrix*}[r]
-4 & 3 \\
1 & -1
\end{pmatrix*}
\begin{pmatrix*}[r]
-2 \\
1 
\end{pmatrix*} 
= 
\begin{pmatrix*}[r]
11 \\
-3 
\end{pmatrix*}
=
(11H - 3E)
\]
which is not divisible in $\WDiv(X)$. Thus $r = 1$.

\noindent\underline{Case $2$}:
Similarly, we have $6\WDiv(Y) \subseteq \CDiv(Y)$. We have
\[
rp^*(6A) = p^*(6rA) = p^*(-6K_Y) = -6K_{Y_0} + 3T
\]
and so 
\[
r\chi_*^{-1}(p^*(6A)) = 
\begin{pmatrix*}[r]
-4 & 3 \\
1 & -1
\end{pmatrix*}
\begin{pmatrix*}[r]
-6 \\
3 
\end{pmatrix*} 
= 
\begin{pmatrix*}[r]
33 \\
-9
\end{pmatrix*}
=
(33H - 9E),
\]
which is divisible by $3$ in $\WDiv(X)$. Thus $r = 1$ or $3$. However, the divisor $A = p_*(\chi_*(11H - 3E))$  has the property $-K_Y = 3A$, thus the index is $3$.

\noindent\underline{Case $3$}: In the final case we have $2\WDiv(Y)\subseteq \CDiv(Y)$ and
\[
rp^*(2A) = p^*(2rA) = p^*(-2K_Y) = -2K_{Y_0} + 2T.
\]
We thus get
\[
r\chi_*^{-1}(p^*(2A)) = 
\begin{pmatrix*}[r]
-4 & 3 \\
1 & -1
\end{pmatrix*}
\begin{pmatrix*}[r]
-2 \\
2 
\end{pmatrix*} 
= 
\begin{pmatrix*}[r]
14\\
-4
\end{pmatrix*}
=
(14H - 4E)
\]
which is divisible by $2$ in  $\WDiv(X)$. As before we may conclude that the index is $2$.
%Link 1: The \textbf{Graded Ring Database} provides a unique answer for $Y$ given the invariants above,
%\href{http://www.grdb.co.uk/search/fano3?fano_index_cmp=eq&fano_index=1&Basket_cmp=eq&Basket=3x1/2(1,1,1)&Basket_size_cmp=eq&Basket_size=3&Degree_cmp=eq&Degree=19/2}{namely}.
%
%
%Link 2: The \textbf{Graded Ring Database} provides a unique answer for $Y$ given the invariants above,
%\href{http://www.grdb.co.uk/search/fano3?fano_index_cmp=eq&fano_index=1&Basket_size_cmp=eq&Basket_size=2&Degree_c%mp=eq&Degree=55/6}{namely}.

%Link 3: There is a unique hit in the \textbf{Graded Ring Database}, %\href{http://www.grdb.co.uk/search/fano3?fano_index_cmp=eq&fano_index=1&Basket_size_cmp=eq&Basket_size=1&Degree_c%mp=eq&Degree=25/2}{namely}.

%Link 4: We again get a unique hit in the \textbf{Graded Ring Database}, %\href{http://www.grdb.co.uk/search/fano3?fano_index_cmp=eq&fano_index=1&Basket_cmp=eq&Basket=2x1/2(1,1,1)&Degree_%cmp=eq&Degree=9}{namely}.

We summarize the data in the following table.

\begin{table}[!h]	
	\rowcolors{3}{gray!25}{white}
	$
	\def\arraystretch{1.2}
	\begin{array}{| c | c | c | c | c |}
	\hline
	\rowcolor{gray!50}
	\# & \text{\begin{tabular}{c}
		Type of the\\ contraction $p$
		\end{tabular} }& 
	\text{ \begin{tabular}{c}
		Singularities \\ of $Y$
		\end{tabular}
	}  & -K_Y^3 &  
	\text{\begin{tabular}{c}
		Fano-Weil \\ Index
		\end{tabular}}  \\
	\hline
	1 &E5& 3\times \frac{1}{2}(1,1,1) 					& \frac{19}{2}  &  1 \\ 
	2 &E5& \frac{1}{2}(1,1,1), \frac{1}{3}(1,1,2) &  \frac{55}{6}  &  3 \\ 
	3 &E3& \frac{1}{2}(1,1,1), \text{odp}								&  \frac{25}{2}  &  2 \\ 
	4 &E5& 2\times \frac{1}{2}(1,1,1)  			&  9  					&  1 \\
	\hline
	\end{array}
	$
\end{table}

%\textcolor{Maroon}{Is there some problem? The cases with Index $>1$ do not appear in the Graded Ring Database. However, the rest of the invariants provide a unique answer in the graded ring database but of index $1$.}

\subsection{Some invariants of the targets of the links - Cases 5 and 6.}

In cases $5$ and $6$ of the table, a similar argument as Proposition \ref{4-secant lines} together with Remark \ref{sum of lines} shows that
after anti-flipping the $5$-secant lines we have the flop of any $4$-secant lines. Using again Remark \ref{sum of lines} we see that there are no irreducible curves between (the rays spanned by the classes of) the $4$-secants and the $3$-secants and so the next step in the link is the contraction of the $3$-secant lines. This is given by the linear system of cubics containing $C$, thus the Sarkisov diagrams take respectively the forms\\
\begin{minipage}{0.5\linewidth}
	\[
	\xymatrix@R=0.5cm@C=0.7cm{
		X \ar[d] \ar@{..>}[rr]^{\text{anti-flip}} && X_0 \ar@{..>}[rr]^{\text{flop}} && Y_0 \ar[d]^p\\
		\PP^3 \ar[rrd]&&&& \PP^1 \ar[lld]\\
		&& pt
	}
	\]
\end{minipage}
\begin{minipage}{0.5\linewidth}
	\[
	\xymatrix@R=0.5cm{
		X \ar[d] \ar@{..>}[rr]^{\text{anti-flip}} && Y_0 \ar[d]^p\\
		\PP^3 \ar[rd]&& \PP^1 \ar[ld]\\
		& pt
	}
	\]
\end{minipage}
where $p\colon Y_0 \map \PP^1$ is a del-Pezzo fibration.

\subsubsection{The fibration {$Y_0 \to \PP^1$}} As in the previous section, the restriction of $\chi\colon X \psmap Y_0$ to any cubic surface can be checked to be the contraction of the anti-flipped/flopped curve. We conclude that:
\begin{itemize}
	\item[$\#5$.] In case $5$, the restriction of $\chi$ on a cubic is the contraction of $2$ curves. The strict transform of any cubic is thus a del-Pezzo surface of degree $5$. 
	\item[$\#6$.] In case $6$, the restriction of $\chi$ on a cubic is the contraction of $1$ curve. The strict transform of any cubic is thus a del-Pezzo surface of degree $4$.
\end{itemize}

\subsubsection{The singularities of {$Y_0$}} As before we will compute the singularities of $Y_0$ following the Sarkisov diagram clockwise.
\begin{itemize}
	\item[$\#5$.] In case $5$, we have the $(1,2)$-flip of $1$ curve followed by the flop of $1$ curve. These pseudo-isomorphisms produce $1$ quotient singularity of type $\frac{1}{2}(1,1,1)$.
	\item[$\#6$.] In case $6$, we have the $(1,2)$-flip of $1$ curve which again produces $1$ quotient singularity of type $\frac{1}{2}(1,1,1)$.
\end{itemize}

\subsubsection{Cube of {$-K_{Y_0}$}} Once more we will use Lemma \ref{cube anticanonical} and Corollary \ref{cube of the anticanonical, flip} and follow the diagram clockwise to compute $(-K_{Y_0})^3$. We have
\begin{itemize}
	\item[$\#5$.] $-K_X^3 = 14\implies -K_{X_0}^3 = \frac{29}{2}\,\,\, \implies -K_{Y_0}^3 = \frac{29}{2}$;
	\item[$\#6$.] $-K_X^3 = 12\implies -K_{X_0}^3 = \frac{25}{2}\,\,\, \implies -K_{Y_0}^3 = \frac{25}{2}$.
\end{itemize}

\begin{table}[!h]	
	\rowcolors{3}{gray!25}{white}
	$
	\def\arraystretch{1.2}
	\begin{array}{| c | c | c | c |}
	\hline
	\rowcolor{gray!50}
	\# & \text{\begin{tabular}{c}
		Type of the\\ contraction $p$
		\end{tabular} }& 
	\text{ \begin{tabular}{c}
		Singularities \\ of $Y$
		\end{tabular}
	}  & -K_Y^3  \\
	\hline
	5 &
	\text{del-Pezzo fibration
		of degree 5}
	& 
	\frac{1}{2}(1,1,1) 					& \frac{29}{2}   \\ 
	6 &
	\text{del-Pezzo fibration	of degree 4}
	& \frac{1}{2}(1,1,1) &  \frac{25}{2}  \\ 
	\hline
	\end{array}
	$	
\end{table}

\bibliographystyle{alpha}

\end{document}